\documentclass[12pt,letterpaper]{amsart}
\usepackage[utf8]{inputenc}
\usepackage{amsmath}
\usepackage{amssymb}
\usepackage{amsfonts}
\usepackage{amsthm}
\usepackage{tikz}
\usepackage{cleveref}
\usepackage{subcaption}
\usepackage{stmaryrd}
\usepackage{wrapfig}
\usepackage[foot]{amsaddr}
\usepackage{url}
\usepackage{enumitem}
\usepackage{todonotes}
\usepackage[normalem]{ulem}
\usepackage{mathtools}
\usepackage{soul}
\usepackage{fullpage}

\newcommand{\bolda}{\mathbf{a}}
\newcommand{\boldb}{\mathbf{b}}
\newcommand{\boldc}{\mathbf{c}}
\newcommand{\bolde}{\mathbf{e}}

\newcommand{\boldx}{\mathbf{x}}

\newcommand{\boldalpha}{\boldsymbol{\alpha}}
\newcommand{\boldbeta}{\boldsymbol{\beta}}
\newcommand{\boldgamma}{\boldsymbol{\gamma}}
\newcommand{\boldzero}{\mathbf{0}}
\newcommand{\boldone}{\mathbf{1}}
\newcommand{\inner}[2]{\langle \,#1\,,\, #2\, \rangle}
\newcommand{\biginner}[2]{\left\langle \,#1\,,\, #2\, \right\rangle}

\newcommand{\R}{\mathbb{R}}
\newcommand{\Z}{\mathbb{Z}}
\newcommand{\subdiv}{\mathcal{D}}
\newcommand{\adjp}{\nabla}
\newcommand{\facets}{\mathcal{F}}
\newcommand{\nodes}{\mathcal{V}}
\newcommand{\edges}{\mathcal{E}}
\newcommand{\term}[1]{\textbf{#1}}
\newcommand{\digraph}[1]{\vec{#1}}

\DeclareMathOperator{\conv}{conv}
\DeclareMathOperator{\evol}{vol}
\DeclareMathOperator{\nvol}{nvol}
\DeclareMathOperator{\rank}{rank}
\DeclareMathOperator{\corank}{corank}

\DeclarePairedDelimiter{\ceil}{\lceil}{\rceil}

\newtheorem{theorem}{Theorem}
\newtheorem{proposition}{Proposition}
\newtheorem{lemma}{Lemma}
\newtheorem{corollary}{Corollary}
\newtheorem{remark}{Remark}

\theoremstyle{definition}
\newtheorem{definition}{Definition}

\title{Graph edge contraction and subdivisions for adjacency polytopes}
\author{Tianran Chen} 
\author{Evgeniia Korchevskaia}
\email{ti@nranchen.org}
\email{ekorchev@aum.edu}
\address{Department of Mathematics, Auburn University Montgomery, Montgomery, AL, USA}
\thanks{The authors are supported by 
a grant from the Auburn University at Montgomery Research Grant-in-Aid Program
and the National Science Foundation under Grant No. 1923099.
EK is also supported by the Undergraduate Research Experience program funded by the Department of Mathematics at Auburn University at Montgomery}

\begin{document}

\begin{abstract}
    Adjacency polytopes, a.k.a. symmetric edge polytopes, associated with undirected graphs
    have been defined and studied in several seemingly independent areas
    including number theory, discrete geometry, and dynamical systems. 
    In particular, the authors are motivated by the 
    tropical intersections problem derived from the Kuramoto equations.
    Regular subdivisions of adjacency polytopes are instrumental in solving these problems.
    This paper explores connections between the regular subdivisions of an
    adjacency polytope and the contraction of the underlying graph along an edge.
    We construct a special regular subdivision
    whose cells are in one-to-one correspondence with facets of an adjacency polytope
    associated with an edge-contraction of the original graph.
    Moreover, this subdivision induces a decomposition of the original graph
    into ``cell subgraphs''.
    We explore the combinatorial, graph-theoretic, and matroidal aspects of this connection.
\end{abstract}

\maketitle

\section{Introduction}

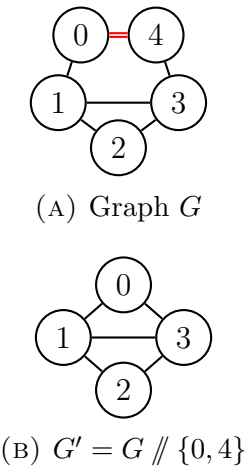
\begin{wrapfigure}[18]{r}{0.26\textwidth}
    \centering
    \begin{subfigure}[b]{0.25\textwidth}
        \centering
        \begin{tikzpicture}[
            every node/.style={circle,thick,draw},
            every edge/.style={draw,thick}]
            \node (0) at ( 0  , .9) {0};
            \node (4) at ( 1  , .9) {4};
            \node (1) at (-0.3,  0) {1};
            \node (3) at ( 1.3,  0) {3};
            \node (2) at ( 0.5,-.6) {2};
            \path (0) edge (1);
            \path (0) edge[color=red,style={double}] (4);
            \path (1) edge (2);
            \path (2) edge (3);
            \path (3) edge (4);
            \path (1) edge (3);
        \end{tikzpicture}
        \caption{Graph $G$}
        \label{fig:one-G}
    \end{subfigure}
    \vspace{2ex}
    
    \begin{subfigure}[b]{0.25\textwidth}
        \centering
        \begin{tikzpicture}[
            every node/.style={circle,thick,draw},
            every edge/.style={draw,thick}]
            \node (0) at ( 0.5  , .7) {0};
            \node (1) at (-0.3,  0) {1};
            \node (3) at ( 1.3,  0) {3};
            \node (2) at ( 0.5,-.7) {2};
            \path (0) edge (1);
            \path (1) edge (2);
            \path (2) edge (3);
            \path (3) edge (0);
            \path (1) edge (3);
        \end{tikzpicture}
        \caption{$G' = G \sslash \{0,4\}$}
        \label{fig:one-G-shrink}
    \end{subfigure}
    \caption{Edge contraction of a graph}
    \label{fig:contraction}
\end{wrapfigure}
For a connected graph $G$ with nodes $\nodes(G) = \{0,1,\dots,n\}$
and edge set $\edges(G)$,
its \emph{adjacency polytope}~\cite{Chen2019Unmixing} 
(a.k.a. \emph{symmetric edge polytope} \cite{Matsui2011Roots}) 
is the convex polytope
    $\adjp_G =
    \conv \{ \bolde_i - \bolde_j \mid \{i,j\} \in \edges(G) \}$.
In the context of Kuramoto models~\cite{Kuramoto1975Self},
the geometric structure of adjacency polytopes turned out to be
instrumental in understanding the root counting problem of algebraic Kuramoto equations \cite{ChenDavis2018Toric, ChenDavisMehta2018Counting,Kuramoto1975Self}.
In the broader context, the adjacency polytope of a graph 
is equivalent to the symmetric edge polytope which has been studied 
by number theorists, combinatorialists, and discrete geometers
motivated by several seemingly independent problems 
\cite{DelucchiHoessly2016Fundamental, higashitani2019ARITHMETIC, Higashitani2016Interlacing,  Matsui2011Roots, ohsugi2014, Ohsugi2012, Rodriguez2002}.
These different viewpoints are consolidated in the recent work
by D'Alì, Delucchi, and Michałek~\cite{dal2019faces},
which, among other contributions, shed new light on the
structure of adjacency polytopes associated with graphs
consisting of two subgraphs sharing a single edge.
In particular, using Gr\"obner bases methods, the authors provided explicit formulae for 
the number of facets and the normalized volume 
of adjacency polytopes associated with graphs formed by gluing together two connected bipartite graphs, trees, or cycles. 
In this paper, we pursue this line of inquiry
by considering the effect of a contraction of a graph along an edge
on the corresponding adjacency polytope.
Of particular importance in the study of algebraic Kuramoto equations
derived from a graph $G$ are the facets and regular subdivisions of 
the corresponding adjacency polytope $\adjp_G$.
The set of facets of $\adjp_G$ corresponds to 
the set of directed acyclic subgraphs of $G$ 
which satisfy certain minimal flow property~\cite{Chen2019Directed}. 
Geometric structure of the facets played a key role in computing
the volume of the adjacency polytope for certain families of graphs~\cite{ChenDavisMehta2018Counting,dal2019faces}
as well as determining the generic root count for algebraic Kuramoto equations~\cite{Baillieul1982,Chen2019Unmixing}.
Equally important, a nontrivial regular subdivision of $\adjp_G$ gave rise to a 
toric deformation of the underlying algebraic Kuramoto equations 
into a simpler system of equations whose solutions can be identified 
with all the complex solutions to the original system~\cite{ChenDavis2018Toric}.
In this paper, we explore connections between the regular subdivision of $\adjp_G$
and the facets of $\adjp_{G \sslash e}$ where 
$G \sslash e$ is the contraction of $G$ along an edge $e \in \edges(G)$. 
See an example in \Cref{fig:contraction}.

The main contribution of this paper is twofold.
First, we construct a special regular subdivision
of $\adjp_G$ whose cells are in one-to-one correspondence with
facets of $\adjp_{G \sslash e}$ associated with the graph $G \sslash e$.
We also show that if $G$ consists of two subgraphs $G_1$ and $G_2$
sharing exactly one edge $e$, 
then the cells in the special regular subdivision of $\adjp_G$ are in 
one-to-one correspondence with the products of facets of the 
adjacency polytopes $\adjp_{G_1'}$ and $\adjp_{G_2'}$, 
where $G_1' = G_1 \sslash e$ and $G_2' = G_2 \sslash e$.
Since this study is largely motivated by the tropical intersections problem
derived from the algebraic Kuramoto equations,
the resulting subdivision of the polytope $\adjp_G$
can be also conveniently viewed as a subdivision of the underlying point configuration.
Combined with the existing knowledge about facets of adjacency polytopes
(symmetric edge polytopes) associated with 
trees, cycles, bipartite graphs, and wheel graphs~\cite{dal2019faces},
this result enables us to study more complicated graphs 
formed by gluing these basic building blocks along the edges.
Second, we show that the resulting subdivision corresponds to 
a decomposition of the original graph into a collection of ``cell subgraphs'',
and we explore the combinatorial, graph-theoretic, and matroidal aspects 
of this correspondence.

This paper is structured as follows.
\Cref{sec:prelim} reviews necessary definitions and notations.
\Cref{sec:subdiv} defines the special regular subdivision
induced by an edge contraction and establishes the correspondence between
cells in this subdivision and facets of smaller adjacency polytopes.
\Cref{sec:properties} explores the correspondence between cells in this subdivision
and their corresponding ``cell subgraphs''.
\Cref{sec:matroid} provides a matroidal interpretation of the symmetry
between cells and the cell subgraphs.
In \Cref{sec:examples}, we show a few concrete examples.
We conclude in \Cref{sec:conclusion}.

\section{Preliminaries and notations}\label{sec:prelim}

Given a (undirected) simple graph $G$,
$\nodes(G)$ and $\edges(G)$ denote the sets of vertices and edges respectively. 
We use the notation $\{a,b\}$
for the (undirected) edge connecting vertices $a$ and $b$. 
With respect to an edge $e = \{a,b\}$,
the \emph{contraction} $G \sslash e$ of $G$ along $e$ is 
a simple graph obtained by merging the vertices $a$ and $b$ in $G$,
i.e.,
$G \sslash e = (V',E')$ with
$V' = \nodes(G) \setminus \{a,b\} \cup \{a'\}$ and
$E' = \edges(G) \setminus \{\{a,b\}\} \cup \{ \{ a', v \} \mid v \ne a,b, \{a,v\} \text{ or } \{b,v\} \text{ is in } \edges(G) \}$.
We also extend this notation to subgraphs of $G$.
For a subgraph $H$ of $G$ and an edge $e \in \edges(G)$, 
we let $H \sslash e$ be the contraction as defined above
if $e \in \edges(H)$,
and just let $H \sslash e = H$ otherwise.

A \emph{convex polytope} is the convex null of a finite set of points.
Its \emph{dimension} is the dimension of the smallest affine space that contains it.
A (nonempty) \emph{face} of a convex polytope is a subset of the polytope on which 
a linear functional $\inner{\cdot}{\boldalpha}$ is minimized.
In this case, $\boldalpha$ is an inner normal vector of the face.
Faces are themselves polytopes, 
and proper faces of the maximal dimension are called \emph{facets}.
In this paper, we only deal with (convex) \emph{lattice polytopes}, i.e., the convex polytopes whose vertices have integer coordinates.
For an $n$-dimensional lattice polytope $P \subset \R^n$, 
its \emph{normalized volume}, denoted by $\nvol(P)$, is $n! \evol(P)$,
which is always an integer.

For a connected graph $G$ with nodes $\nodes(G) = \{0,1,\dots,n\}$
and 
edge set $\edges(G)$,
its \emph{adjacency polytope}~\cite{Chen2019Unmixing} 
(a.k.a. \emph{symmetric edge polytope}~\cite{dal2019faces,Matsui2011Roots}) 
is the convex polytope
\begin{equation}
    \adjp_G =
    \conv \{ \bolde_i - \bolde_j \mid \{i,j\} \in \edges(G) \}
    \subset \R^n,
\end{equation}
where $\bolde_i \in \R^n$ is the vector with 1 in the $i$-th entry and zero elsewhere, 
and $\bolde_0 = \boldzero$. 

In the trivial case of $n=0$, $\adjp_G$ is simply $\{ \boldzero \}$,
and we adopt the convention that the only facet of $\adjp_G$ is $\varnothing$. 
For $n > 0$, $\adjp_G$ is a full-dimensional polytope in $\R^n$.

The set of facets of $\adjp_G$ is denoted by $\facets(\adjp_G)$. 
Any facet of $\adjp_G$ is an intersection of this polytope with a supporting hyperplane, 
which is uniquely determined by an inner normal vector \cite{Ziegler1995Lectures}. 
Moreover, by construction, $\boldzero$ is an interior point of $\adjp_G$,
which allows the inner normal vectors to be normalized to a certain form.
We state this observation as a lemma for later reference.

\begin{lemma}\label{lem:facet-def}
    For any graph $G$, 
    a nonzero vector $\boldalpha$ defines a facet of $\adjp_G$
    if and only if there are $x_1,\dots,x_n \in \adjp_G$ such that
    $x_1,\dots,x_n$ are linearly independent as vectors and
    \begin{align*}
        \inner{ \boldx_i }{ \boldalpha } &=   -1 \quad\text{for any } i=1,\dots,n, \text{and}\\
        \inner{ \boldx }{ \boldalpha }   &\ge -1 \quad\text{for any } \boldx \in \adjp_G. \\
    \end{align*}
\end{lemma}


A (polyhedral) \emph{subdivision} of a convex polytope $P$
is a collection $\subdiv$ of convex polytopes contained in $P$ and of the same dimension as $P$ 
such that their union is $P$ and the intersection of any two
is their (possibly empty) common face.
Elements of a subdivisions are known as cells.
A \emph{point configuration} is a finite collection of labeled points $S \subset \R^n$ \cite{loera_triangulations:_2010}. A subdivision of $S$ is simply a subdivision of $\conv(S)$
whose cells are convex hulls of the subsets of $S$.
For such a cell $C$, we use the notations
$\dim(C) := \dim(\conv(C))$,
$\evol(C) := \evol(\conv(C))$ and
$\nvol(C) := \nvol(\conv(C))$.

Regular subdivision is a particularly important class of subdivisions.
For a point configuration $S$,
using weights assigned by a function $\omega: S \to \R$, we define
$\hat{S} = \{ (\boldx, \omega(\boldx) ) \mid \boldx \in S \}$.
An inner normal vector $\hat{\boldalpha} \in \R^{n+1}$ 
of a face of $\conv(\hat{S})$ is said to be \emph{upward pointing} 
if $\inner{ \bolde_{n+1} }{ \hat{\boldalpha} } > 0$.
A facet of $\conv(\hat{S})$ with an upward pointing inner normal vector
is called a \emph{lower facet}.
The projection of all lower facets of $\conv({\hat{S}})$ form a subdivision of $S$,
the \emph{regular subdivision} (a.k.a. coherent subdivision) 
of $S$ induced by weight function $\omega$ ~\cite{gelfand_discriminants_1994, loera_triangulations:_2010}.
In this case, a lower facet is defined by a vector $\boldalpha \in \R^n$,
a value $h \in \R$ and 
a set $C \subset S$ with $|C| \ge n$, $\dim(\conv(C)) = n$ such that
\begin{equation}\label{equ:lower-facet}
    \begin{aligned}
        \inner{ \boldx }{ \boldalpha } + \omega(\boldx) &= h &&\text{for all } \boldx \in C, \\
        \inner{ \boldx }{ \boldalpha } + \omega(\boldx) &> h &&\text{for all } \boldx \in S \setminus C.
    \end{aligned}
\end{equation}
The construction of a special regular subdivision of an adjacency polytope
induced by an edge contraction of the underlying graph is the main focus of this paper.

\section{Regular subdivision induced by edge contraction}\label{sec:subdiv}

As noted in the definition of adjacency polytopes,
we identify edges of a graph of $n+1$ nodes with points in $\R^n$ via the map
\begin{equation}\label{equ:phi}
    \phi((i,j)) =
    \bolde_i - \bolde_j
\end{equation}
and consider an undirected edge $\{i,j\} \in \edges(G)$
as a pair of directed edges $(i,j)$ and $(j,i)$.
With this, the adjacency polytope of $G$ is simply
$\conv(\phi(\edges(G)))$.
In this section, we construct a regular subdivision of $\adjp_G$
induced by an edge contraction.

\begin{definition} \label{def:edge-subdiv} [Edge contraction subdivision]
    For an edge $\{k_1,k_2\} \in \edges(G)$ to be contracted,
    we define the lifting function 
    $\omega_{k_1,k_2} : \phi(\edges(G)) \to \Z$ given by
    \begin{equation}\label{equ:weights}
        \omega_{k_1,k_2}(\bolde_i - \bolde_j) = 
        \begin{cases}
            0 & \text{if } \{ i,j \} = \{ k_1, k_2 \}, \\
            1 & \text{otherwise},
        \end{cases}
    \end{equation}
    and the resulting lifted polytope
    \begin{equation}
        \hat{\adjp}_G =
        \conv \{ 
            (\bolde_i - \bolde_j, \omega_{k_1,k_2}(\bolde_i - \bolde_j)) 
            \;\mid\; 
            \{i,j\} \in \edges(G) 
        \}
        \subset \R^{n+1}.
    \end{equation}
    The projections of the facets of the lower hull of $\hat{\adjp}_{G}$
    onto $\R^n \times \{0\}$ form a subdivision of $\adjp_G$, 
    the regular subdivision induced by $\omega_{k_1,k_2}$.
    This subdivision, denoted by $\subdiv_{k_1,k_2}$ will be referred to as the
    \emph{edge contraction subdivision} of $\adjp_G$
    induced by the edge contraction of $G$ along $\{k_1,k_2\}$.
\end{definition}

\begin{remark}\label{rmk:refnode}
    In the following discussion,
    we will make frequent use of an observation
    that can simplify our notation and calculation.
    Since the choice of reference node is arbitrary,
    without loss of generality and after re-indexing the nodes,
    we can assume $\{0,k\}$ is the shared edge of $G_1$ and $G_2$
    for some $k \ne 0$.
    This corresponds to a projection of the symmetric edge polytope
    onto one of the coordinate planes.
\end{remark}

\begin{lemma}\label{lem:special-edge}
    For a connected graph $G$ and one of its edges, $\{ k_1, k_2 \}$,
    every cell in the edge contraction subdivision $\subdiv_{k_1,k_2}$
    must contain both 
    $\bolde_{k_1} - \bolde_{k_2}$ and
    $\bolde_{k_2} - \bolde_{k_1}$.
\end{lemma}

\begin{proof}
    As noted in~\Cref{rmk:refnode}, without loss of generality, 
    we can assume $\{k_1,k_2\} = \{0,k\}$ for some $k \ne 0$.
    In the follow, we consider the regular subdivision induced by $\omega = \omega_{0,k}$
    and will show that $\pm \bolde_k \in C$ for all $C \in \subdiv$.
    
    Fix a cell $C \in \subdiv$, 
    let $\hat{C}$ be the corresponding lower facet of $\hat{\adjp}_G$,
    let $\hat{\boldalpha} = (\boldalpha,1) = (\alpha_1,\dots,\alpha_n,1)$ 
    be the upward pointing inner normal vector of $\hat{C}$,
    and let $h = \min \{ \inner{\hat{\boldalpha}}{ \hat{\boldx} } \mid \hat{\boldx} \in \hat{\adjp}_G \}$.
    
    Suppose $\pm \bolde_k \not\in C$,
    then there is a set of $n+1$ affinely independent points 
    of the form $\bolde_i - \bolde_j$ with $\{ i,j \} \ne \{0,k\}$ in $C$.
    By assumption, $\hat{\boldalpha}$ is orthogonal to the affine span of this set.
    However, since $\omega(\bolde_i-\bolde_j) = 1$ for $\{ i,j \} \ne \{0,k\}$,
    the affine span of $\hat{C}$ must be $\{ (\boldx,1) \mid \boldx \in \R^n \}$,
    and consequently its normal vector $\hat{\boldalpha}$ must be $(0,\dots,0,1)$.
    Then for any $\bolde_i - \bolde_j \in C$,
    \[
        \inner{ (\bolde_i-\bolde_j, \omega(\bolde_i-\bolde_j)) }{ \hat{\boldalpha} } 
        \; = \; 1 \; > \; 0 \; = \;
        \inner{ (\pm \bolde_k, \omega(\pm \bolde_k ) ) }{ \hat{\boldalpha} },
    \]
    contradicting with the assumption that
    $\inner{ \hat{\boldalpha} }{ \cdot }$ minimizes on $\hat{C}$ over $\hat{\adjp}_G$.
    We can conclude then either $\bolde_k$ or $-\bolde_k$ must be in $C$.
    
    Now suppose $\bolde_k \in C$ but $-\bolde_k \not\in C$.
    then $\inner{-\bolde_k}{\boldalpha} > \inner{\bolde_k}{\boldalpha}$
    which implies that $h = \inner{\bolde_k}{\boldalpha} < 0$.
    Since $C$ is an $n$-dimensional cell,
    there is a set $\Delta$ of $n$ affinely independent points in $C$ 
    of the form $\bolde_i - \bolde_j$ with $\{i,j\} \ne \{0,k\}$,
    i.e., $\{ \bolde_i - \bolde_j - \bolde_k \mid \bolde_i-\bolde_j \in \Delta \}$ 
    is a linearly independent set.
    Let $A$ be the $n \times n$ matrix whose rows are points in $\Delta$
    as row vectors,
    and let $B = A - \boldone \bolde_k^\top$, which is nonsignular.
    Recall that $C$ is the projection of a lower facet of $\hat{\adjp}_G$
    defined by the inner normal vector $\hat{\boldalpha} = (\boldalpha,1)$.
    Therefore,
    \[
        \inner{ \bolde_i - \bolde_j }{ \boldalpha } + 1 = \inner{ \bolde_k }{ \boldalpha }
        \quad\text{for each } \bolde_i - \bolde_j \in \Delta,
    \]
    which is equivalent to
    \[
        B \boldalpha = A \boldalpha - \boldone \bolde_k^\top \boldalpha = -\boldone.
    \]
    We will show this contradicts with the assumption that $h < 0$.
    
    Suppose $A$ is singular, let $\boldx$ be a nonzero vector in its null space.
    Then $\bolde_k^\top \boldx \ne 0$, 
    since $B \boldx = A \boldx - \boldone \bolde_k^\top \boldx$ cannot be zero.
    We can verify that $\boldalpha = \boldx / \bolde_k^\top \boldx$,
    and thus 
    \[
        h = \inner{ \bolde_k }{\boldalpha } = \bolde_k^\top \boldx / \bolde_k^\top \boldx = 1,
    \]
    which contradicts with the assumption that $h < 0$.
    
    On the other hand, if $A$ is nonsingular, then without loss of generality,
    it is possible to re-index the nodes $\{0,1,\dots,n\} \setminus \{0,k\}$
    so that for each $i \ne 0,k$, $\pm (\bolde_i - \bolde_j) \in \Delta$ implies $j>i$.
    With this arrangement, $A$ is upper triangular and its diagonal entries are $\pm 1$.
    Therefore, $A$ is unimodular. 
    Consequently, $A^{-1}$ exists and is an integer matrix.
    Recall that $B \boldalpha = A \boldalpha - \boldone \bolde_k^\top \boldalpha = -\boldone$,
    and $h = \bolde_k^\top \boldalpha$.
    This equation can be written as
    \begin{align*}
        A \, \boldalpha &= (h - 1) \, \boldone, &
        &\text{i.e.} &
        \boldalpha &= (h - 1) \, A^{-1} \, \boldone,
    \end{align*}
    which gives us the relation
    \[
        h = \bolde_k^\top \, \boldalpha = h (\bolde_k^\top \, A^{-1} \boldone) - \bolde_k^\top A^{-1} \boldone
        = y \, h - y
    \]
    if we let $y= \bolde_k^\top A^{-1} \boldone$.
    The above equation implies that $y \ne 1$.
    Moreover, since $A^{-1}$ is an integer matrix, $y \in \Z$.
    Therefore, 
    \[
        h = \frac{y}{y-1} \ge 0
    \]
    contradicting with the assumption that $h = \bolde_k^\top \boldalpha < 0$.
    That is, the assumption $\bolde_k \in C$ but $\-\bolde_k \not\in C$
    leads to a contradiction.
    We can therefore conclude that $\bolde_k \in C$ implies $-\bolde_k \in C$.
    By the same argument,
    it can be shown that $-\bolde_k \in C$ implies $\bolde_k \in C$.
    Hence, $\pm \bolde_k \in C$. 
\end{proof}

\begin{corollary} \label{Cor: h=0}
    Consider contraction of $G$ along the edge $\{k_1,k_2\}$.
    Given a cell $C \in \subdiv$, 
    let $\hat{C}$ be the corresponding lower facet of $\hat{\adjp}_G$. 
    If $\hat{\boldalpha} = (\boldalpha,1) = (\alpha_1,\dots,\alpha_n,1)$ 
    is an upward pointing inner normal vector of $\hat{C}$, 
    then $\inner{ \hat{\boldx} }{\hat{\boldalpha}} = 0$
    for any $\hat{\boldx} \in \hat{C}$. 
    In particular, 
    \begin{align*}
        \inner{ {\boldx} }{{\boldalpha} }
        &=
        \begin{cases}
            0 & \text{if } \{ i,j \} = \{ k_1, k_2 \}, \\
            -1 & \text{otherwise},
        \end{cases}
        &&\text{and} &
        \alpha_{k_1} = \alpha_{k_2}.
    \end{align*}
\end{corollary}

In the following, using the special edge contraction subdivision,
we establish the link between $\adjp_G$ and $\adjp_{G \sslash \{k_1,k_2\}}$.

\subsection{Two subgraphs sharing an edge}\label{subsection: Two subgraphs}

\begin{figure}
    \centering
    \begin{subfigure}[b]{0.33\textwidth}
        \centering
        \begin{tikzpicture}[
            every node/.style={circle,thick,draw,inner sep=2.5},
            every edge/.style={draw,thick}]
            \node (0) at ( 0,1) {0};
            \node (1) at (-1,1) {1};
            \node (2) at (-1,0) {2};
            \node (3) at ( 0,0) {3};
            \node (4) at (0.85,-.2) {4};
            \node (6) at (0.85,1.2) {6};
            \node (5) at (1.5,0.5) {5};
            \path (0) edge (1);
            \path (0) edge[color=red,style={double}] (3);
            \path (1) edge (2);
            \path (0) edge (2);
            \path (2) edge (3);
            \path (0) edge (6);
            \path (3) edge (4);
            \path (4) edge (5);
            \path (6) edge (5);
            \path (5) edge (3);
            \path (5) edge (0);
        \end{tikzpicture}
        \caption{Graph $G$ formed by two subgraphs sharing an edge}
        \label{fig:G}
    \end{subfigure}
    ~
    \begin{subfigure}[b]{0.3\textwidth}
        \centering
        \begin{tikzpicture}[
            every node/.style={circle,thick,draw,inner sep=2.5},
            every edge/.style={draw,thick}]
            \node (0) at ( 0,0.5) {0};
            \node (1) at (-1,1) {1};
            \node (2) at (-1,0) {2};
            \node (4) at (0.80,-.2) {4};
            \node (6) at (0.80,1.2) {6};
            \node (5) at (1.6,0.5) {5};
            \path (0) edge (1);
            \path (1) edge (2);
            \path (2) edge (0);
            \path (0) edge (6);
            \path (0) edge (4);
            \path (4) edge (5);
            \path (6) edge (5);
            \path (5) edge (0);
            \path (5) edge (0);
        \end{tikzpicture}
        \caption{Graph $G'$ resulted from edge contraction}
        \label{fig:G-shrink}
    \end{subfigure}
    ~
    \begin{subfigure}[b]{0.30\textwidth}
        \centering
        \begin{tikzpicture}[
            every node/.style={circle,thick,draw,inner sep=2.5},
            every edge/.style={draw,thick}]
            \node (0) at ( 0,0.5) {0};
            \node (1) at (-1,1) {1};
            \node (2) at (-1,0) {2};
            \node (3) at ( 1,0.5) {0};
            \node (4) at (1.80,-.2) {4};
            \node (6) at (1.80,1.2) {6};
            \node (5) at (2.6,0.5) {5};
            \path (0) edge (1);
            \path (1) edge (2);
            \path (2) edge (0);
            \path (3) edge (6);
            \path (3) edge (4);
            \path (4) edge (5);
            \path (6) edge (5);
            \path (5) edge (3);
            \path (5) edge (3);
        \end{tikzpicture}
        \caption{Two subgraphs $G_1'$ and $G_2'$ of $G'$}
        \label{fig:subgraphs}
    \end{subfigure}
    \caption{Edge contraction on a graph}
    \label{fig:example1}
\end{figure}
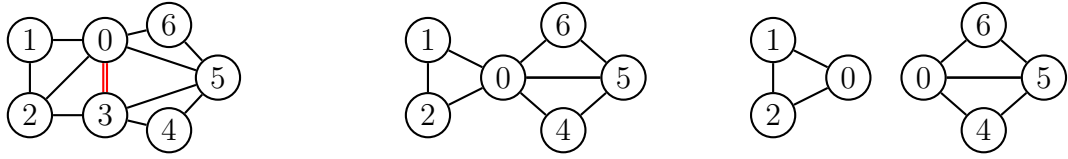

We first consider the case where the target graph $G$
consists of two sub-graphs sharing a single edge
with the two corresponding nodes forming a cut set.
\Cref{fig:G} shows an example of such a graph.
That is, there are two graphs $G_1 = (V_1, E_1)$ and $G_2 = (V_2, E_2)$
such that $\nodes(G) = V_1 \cup V_2, \edges(G) = E_1 \cup E_2$,
and there is one edge $e = \{ k_1, k_2 \} \in \edges(G)$ such that
$V_1 \cap V_2 = \{k_1,k_2\}$ and
$E_1 \cap E_2 = \{ e \}$.
The contraction $G' = G \sslash e$, shown in \Cref{fig:G-shrink}, 
thus has a cut vertex, which allows us to consider the two separate graphs (\Cref{fig:subgraphs}).

\begin{theorem} \label{thrm: 1-1 correspondence}
    For a connected graph $G$ consisting of two subgraphs $G_1$ and $G_2$
    sharing a single edge $e = \{ k_1, k_2 \}$,
    let $G_1' = G_1 \sslash e$ and $G_2' = G_2 \sslash e$.
    Then the cells in the edge contraction subdivision $\subdiv_{k_1,k_2}$ of $\adjp_G$
    induced by the contraction of $G$ along $\{k_1,k_2\}$
    are in one-to-one correspondence
    with pairs in $\facets(\adjp_{G_1'}) \times \facets(\adjp_{G_2'})$.
\end{theorem}

\begin{proof}
    As before, using the observation provided in~\Cref{rmk:refnode}, 
    we can assume $\{0,k\}$ is the shared edge for some $k \ne 0$.
    In addition, we assume the index $k$ is chosen so that
    $i < k$ for all $i \in \nodes(G_1)$ and $k < j$ for all $j \in \nodes(G_2)$.
    That is, after renaming the nodes, we assume 
    nodes $1,2,\dots,k-1$ are in $G_1$,
    nodes $k+1,\dots,n$   are in $G_2$, and
    nodes $0,k$ are in both.
    We only need to consider the subdivision $\subdiv_{0,k}$.
    
    
    By definition, the adjacency polytopes $\adjp_{G_1'}$ and $\adjp_{G_2'}$ are 
    full-dimensional polytopes in the subspaces
    $\R^{k-1} \times \{ \boldzero_{n-k+1} \}$ and
    $\{ \boldzero_{k} \} \times \R^{n-k}$ respectively.
    So their facets are of dimensions $k-2$ and $n-k-1$ respectively.
    
    By~\Cref{lem:special-edge}, any cell $C$ must contain both $\pm \bolde_k$,
    and therefore its upward pointing inner normal vector 
    $\hat{\boldgamma} = (\boldgamma,1) \in \R^{n+1}$
    that defines the lower facet $\hat{C}$ of $\hat{\adjp}_G$
    satisfies $\inner{ \pm \bolde_k }{ \boldgamma } = 0$.
    Thus $\hat{C}$ is contained in the hyperplane $\inner{ \cdot }{ \hat{\boldgamma} } = 0$, and
    \[
        \boldgamma = (\boldalpha, 0, \boldbeta)
        \quad
        \text{for some } \boldalpha \in \R^{k-1}
        \text{ and }     \boldbeta  \in \R^{n-k}.
    \]
    Let $C_1$ and $C_2$ be the sets of projections of points in $C$ in
    $\R^{k-1} \times \{ \boldzero_{n-k+1} \}$ and
    $\{ \boldzero_k \} \times \R^{n-k}$ respectively.
    We will show the nonzero points in $C_1$ and $C_2$ 
    define facets of $\adjp_{G_1'}$ and $\adjp_{G_2'}$ respectively.
    Since $\gamma = (\boldalpha,0,\boldbeta)$,
    the equation \eqref{equ:lower-facet} which defines the lower facet $\hat{C}$ implies
    \begin{align}
        \inner{ \bolda }{            (\boldalpha, \boldzero_{n-k+1}) } &=   -1 
        &&\text{for all } \boldzero \ne \bolda \in C_1, 
        \label{equ:C1} \\ 
        \inner{ \boldb }{            (\boldzero_k,\boldbeta) }         &=   -1 
        &&\text{for all } \boldzero \ne \boldb \in C_2, 
        \label{equ:C2} \\
        \inner{ \bolde_i-\bolde_j }{ (\boldalpha, \boldzero_{n-k+1}) } &\ge -1 
        &&\text{for all } \{i,j\} \in \edges(G_1'), \nonumber
        \\
        \inner{ \bolde_i-\bolde_j }{ (\boldzero_k,\boldbeta) }         &\ge -1 &&\text{for all } \{i,j\} \in \edges(G_2'),
        \nonumber
    \end{align}
    and there are at least $n-1$ equalities in total in this system
    (since the two equalities associated with $\pm \bolde_k$ are removed).
    Moreover, since $\dim(\conv(C)) = n$, 
    the corresponding points can be chosen to be linearly independent as vectors.
    Therefore, there are at least $k-1$ and $n-k$ equalities among
    \eqref{equ:C1} and \eqref{equ:C2} respectively.
    By~\Cref{lem:facet-def}, the vectors 
    $(\boldalpha, \boldzero_{n-k+1})$ and $(\boldzero_k, \boldbeta)$
    must define facets in $\adjp_{G_1'}$ and $\adjp_{G_2'}$ respectively.
    That is, if we identify each cell with its corresponding upward-pointing inner normal vector,
    then the map
    \[
        \gamma
        \;\mapsto\;
        (\boldalpha,\boldbeta)
        \in \R^{k-1} \times \R^{n-k}
    \]
    sends each cell in $\subdiv_{0,k}$ to a pair of facets of 
    $\adjp_{G_1'}$ and $\adjp_{G_2'}$ respectively.
    
    We now simply have to show that this map has an inverse.
    Suppose $F_1'$ and $F_2'$ are two facets of $\adjp_{G_1'}$ and $\adjp_{G_2'}$ respectively.
    We will construct a  corresponding cell in $\subdiv_{0,k}$.
    Let $\bolda_1,\dots,\bolda_{m_1'} \in \phi(\edges(G_1'))$
    and $\boldb_1,\dots,\boldb_{m_2'} \in \phi(\edges(G_2'))$
    be the points defining $F_1'$ and $F_2'$ respectively
    for some $m_1' \ge k-1$ and $m_2' \ge n-k$.
    Then by~\Cref{lem:facet-def}, there are vectors 
    $\boldalpha \in \R^{k-1}$ and $\boldbeta \in \R^{n-k}$ such that
    \begin{equation*}
        \begin{aligned}
            \inner{ \bolda_i }{            (\boldalpha, \boldzero_{n-k+1}) } &= -1 \quad\text{for all } i=1,\dots,m_1', \\
            \inner{ \bolde_i - \bolde_j }{ (\boldalpha, \boldzero_{n-k+1}) } &\ge -1 \quad\text{for all } \{i,j\} \in \edges(G_1'), \\
            \inner{ \boldb_i }{            (\boldzero_k, \boldbeta) } &= -1 \quad\text{for all } i=1,\dots,m_2', \\
            \inner{ \bolde_i - \bolde_j }{ (\boldzero_k, \boldbeta) } &\ge -1 \quad\text{for all } \{i,j\} \in \edges(G_2').
        \end{aligned}
    \end{equation*}
    Define
    \[
        \boldgamma = ( \boldalpha, 0, \boldbeta) \in \R^{k-1 + 1 + n-k} = \R^n.
    \]
    We will verify that $\hat{\boldgamma} = (\boldgamma,1) \in \R^{n+1}$ 
    defines a lower facet of $\hat{\adjp}_G$.
    Let
    \begin{align*}
        C_1 &=
        \left\{
            (\bolda_i, p, \boldzero)
            \;\mid\;
            p \in \{-1,0,1\}
            \;,\;
            i=1,\dots,m_1'
        \right\}
        \;\cap\;
        \phi_0(\edges(G)),
        \\
        C_2 &=
        \left\{
            ( \boldzero, p , \boldb_i )
            \;\mid\;
            p \in \{-1,0,1\}
            \;,\;
            i=1,\dots,m_2'
        \right\}
        \;\cap\;
        \phi_0(\edges(G)).
    \end{align*}
    Then
    \begin{align*}
        \inner{ (\bolda_i , p , \boldzero ) }{ ( \boldalpha , 0 , \boldbeta ) } + 1
        &=
        \inner{\bolda_i}{\boldalpha} + 1
        = 
        -1 + 1
        =
        0
        &&\text{for each }
        ( \bolda_i , p , \boldzero )
        \in C_1,\\
        \inner{  (\boldzero, p, \boldb_i) }{ (\boldalpha , 0 , \boldbeta) } + 1
        &=
        \inner{ \boldb_i }{ \boldbeta } + 1
        = 
        -1 + 1
        =
        0
        &&\text{for each }
        ( \boldzero, p, \boldb_i )
        \in C_2.
    \end{align*}
    Moreover, $\inner{ (\pm \bolde_k,0) }{ \hat{\boldgamma} } = 0$.
    Let
    \[
        C = C_1 \cup C_2 \cup \{ \pm \bolde_k \},
    \]
    then $\hat{C}$ is contained in the hyperplane defined by $\inner{ \cdot }{ \hat{\gamma} } = 0$.
    For points in $\phi(\edges(G)) \setminus \{ \pm \bolde_k \}$,
    direct computation confirms that
    \begin{align*}
        \inner{ \pm (\bolde_i - \bolde_j) }{ \boldgamma } + 1 &=
        \inner{ \pm (\bolde_i - \bolde_j) }{ (\boldalpha, \boldzero_{n-k+1}) } + 1 \ge 0
        &&\text{for any } i,j < k, 
        \\
        \inner{ \pm (\bolde_i - \bolde_j) }{ \boldgamma } + 1 &=
        \inner{ \pm (\bolde_i - \bolde_j) }{ (\boldzero_k, \boldbeta) } + 1 \ge 0
        &&\text{for any } i,j > k, 
        \\
        \inner{ \pm (\bolde_i - \bolde_k) }{ \boldgamma } + 1 &=
        \inner{ \pm \bolde_i }{ (\boldalpha, \boldzero_{n-k+1}) } + 1 \ge 0
        &&\text{for any } 0 < i < k ,
        \\
        \inner{ \pm (\bolde_j - \bolde_k) }{ \boldgamma } + 1 &=
        \inner{ \pm \bolde_j }{ (\boldzero_k, \boldbeta) } + 1 \ge 0
        &&\text{for any } j > k ,
        \\
        \inner{ \pm (\bolde_j - \bolde_0) }{ \boldgamma } + 1 &=
        \inner{ \pm \bolde_j }{ (\boldzero_k, \boldbeta) } + 1 \ge 0
        &&\text{for any } j > k.
    \end{align*}
    Therefore, $\conv(C)$ is a projection of a lower face.
    By construction,
    \[ 
        |C| = |C_1| + |C_2| + | \{ \pm \bolde_k \}| \ge (k-1) + (n-k) + 2 = n + 1, 
    \]
    and these points have affinely indepedent projections in 
    $\R^{k-1} \times \{ \boldzero \}$,
    $\{ \boldzero_{k-1} \} \times \R \times \{ \boldzero_{n-k} \}$, or
    $\{ \boldzero_{k} \} \times \R^{n-k}$.
    So,
    \[ 
        \dim(\conv(C)) \ge k-1 + 1 + n-k = n.
    \]
    Therefore, $C$ must be a projection of a lower facet of $\hat{\adjp}_G$
    and hence a cell in $\subdiv_{0,k}$.
\end{proof}

The theorem above establishes a bijection between cells in $\subdiv_{k_1,k_2}$
and pairs of faces in $\facets(\adjp_{G_1'})$ and $\facets(\adjp_{G_2'})$.
For later reference, this bijection will be denoted by
$q_{k_1,k_2} : \subdiv_{k_1,k_2} \to \facets(\adjp_{G_1'}) \times \facets(\adjp_{G_2'})$
and given by
\[
    q_{k_1,k_2} (C) = (F_1,F_2),
\]
where $F_1$ and $F_2$ are simply the convex hull of the projections of $C$
in the coordinate-subspaces in which $\adjp_{G_1'}$ and $\adjp_{G_2'}$ are full-dimensional.

\begin{remark}
    Note that $q_{k_1,k_2}$ is only a bijection between the set of cells in $\subdiv_{k_1,k_2}$
    and the set $\facets(\adjp_{G_1'}) \times \facets(\adjp_{G_2'})$.
    The points in $C$ themselves may not be in one-to-one correspondence
    with vertices in $F_1$ and $F_2$.
    In general, the projection that maps points in $C$
    to vertices of $F_1$ and $F_2$ may be not be one-to-one.
    This is a reflection of the fact that the edge-contraction operation
    may map multiple edges to the same edge,
    since we only allow simple graphs 
    (graphs with no multiple edges and loops).
\end{remark}

In certain applications (e.g., the root counting problem for algebraic Kuramoto equations),
simplicial cells are of great importance
as they form the minimum building blocks of $\adjp_G$.
We shall show such simplicial cells corresponds to simplicial facets
of the adjacency polytopes derived from the graph contraction.

\begin{theorem} \label{thrm: simplical cell volume}
    Suppose $G$ is a graph consisting of two subgraphs $G_1$ and $G_2$
    sharing a single edge $e = \{ k_1, k_2 \}$,
    with $G_1' = G_1 \sslash e$ and $G_2' = G_2 \sslash e$.
    Let $C$ be a cell in the edge contraction subdivision $\subdiv_{k_1,k_2}$
    with $q_{k_1,k_2}(C) = (F_1,F_2)$ for some facets $F_1$ and $F_2$
    of $\adjp_{G_1'}$ and $\adjp_{G_2'}$ respectively. 
    If $C$ is simplicial, then $F_1$ and $F_2$ are both simplicial. 
\end{theorem}

\begin{proof}
    Without loss of generality, we still adopt the convention that
    $\{0,k\}$ is the shared edge, 
    and $i \le k$ for all $i \in \nodes(G_1)$ and $k \le j$ for all $j \in \nodes(G_2) \setminus \{0\}$.
    With this convention, $\adjp_{G_1'}$ and $\adjp_{G_2'}$ are embedded in
    $\R^{k-1} \times \{ \boldzero_{n-k+1} \}$ and $\{ \boldzero_{k} \} \times \R^{n-k}$ respectively.
    
    Let $C \in \subdiv_{0,k}$ be a simplicial cell,
    and let $(F_1,F_2) = q_{0,k}(C)$,
    then $F_1$ and $F_2$, being facets of $\adjp_{G_1'}$ and $\adjp_{G_2'}$, 
    are of dimensions $k-2$ and $n-k-1$ respectively.
    Suppose either $F_1$ or $F_2$ is not simplicial,
    then the combined total number of vertices is at least
    \[
        (k-2+1) + (n-k-1+1) + 1 = n.
    \]
    Since these points are nonzero projections of points in $C$,
    so $C$ contains convex independent set of $n$ points 
    $\boldx_1,\dots,\boldx_n \in \phi(\edges(G))$
    with nonzero projections in 
    $\R^{k-1} \times \{ \boldzero_{n-k+1} \}$ or 
    $\{ \boldzero_{k} \} \times \R^{n-k} \}$.
    In addition, $C$ contains two points $\pm \bolde_k$,
    which are in the fiber over $\boldzero$ with respect to either projection,
    and thus $\pm \bolde_k \not\in \conv \{ \boldx_1,\dots,\boldx_n \}$.
    Therefore, $C$ contain at least $n+2$ points,
    and none of them is a interior point.
    This contradicts with the assumption that $C$ is simplicial.
    Therefore, we can conclude that if $C$ is simplicial,
    then $F_1$ and $F_2$ must also be simplicial.
\end{proof}

\subsection{Edge contraction in a single graph}

We now apply the results from Section \ref{subsection: Two subgraphs} to the edge contraction in a single graph $G$. 
Namely, we view $G$ as a union of two graphs, $G_1 = G$ and 
$G_2 = ( \{k_1,k_2\}, \{ \{ k_1, k_2 \} \})$. 
Then Theorems \ref{thrm: 1-1 correspondence} and \ref{thrm: simplical cell volume} 
imply the following corollaries.

\begin{corollary}
    Let $e = \{ k_1, k_2 \}$ be an edge of a connected graph $G$, 
    and let $G' = G \sslash e$.
    Then the cells in the edge contraction subdivision $\subdiv_{k_1,k_2}$ of $\adjp_G$
    induced by the contraction of $G$ along the edge $e$
    are in one-to-one correspondence
    with the facets of $\adjp_{G'}$.
\end{corollary}

\begin{corollary} \label{cor: C simplex -> F simplex}
    Let $e = \{ k_1, k_2 \}$ be an edge of a connected graph $G$. Let $G' = G \sslash e$.
    Suppose $C$ is a cell in the edge contraction subdivision $\subdiv_{k_1,k_2}$ 
    with $q_{k_1,k_2}(C) = (F, \varnothing)$ for some facet $F$ of $\adjp_{G'}$.
    If $C$ is simplicial, then $F$ is simplicial. 
\end{corollary}

\section{Properties of cells and cell subgraphs}\label{sec:properties}

Throughout this section, we fix $G$ to be a connected graph,
and let $\subdiv_{k_1,k_2}$ be the edge contraction subdivision
induced by the edge contraction of $G$ along the edge $\{ k_1, k_2 \} \in \edges(G)$.
Since the polytope $\adjp_G$ is derived from the graph $G$,
cells and hence the subsets of cells in the subdivision $\subdiv_{k_1,k_2}$
are naturally associated with the subgraphs of $G$.
This connection provides a great insight into the combinatorial structure
of the subdivision $\subdiv_{k_1,k_2}$,
which is the main focus of this section.

\begin{definition} \label{def: cell subgraphs}
    For a nonempty subset $X$ of $\{ \bolde_i - \bolde_j \mid \{i,j\} \in \edges(G) \}$,
    we define the corresponding directed and undirected subgraphs 
    $\digraph{G}_X$ and $G_X$ to be the graphs with the edge sets 
    \begin{align*}
        \edges(\digraph{G}_X) &= \{ \; (i,j) \;\mid\; \bolde_i - \bolde_j \in X \}, \;\text{and} \\
        \edges(G_X)           &= \{ \; 
            \{ i, j \} 
            \;\mid\; \bolde_i - \bolde_j \in X \text{ or } \bolde_j - \bolde_i \in X 
        \},
    \end{align*}
    respectively.
    In addition, if $C$ is a cell in $\subdiv_{k_1,k_2}$,
    $\digraph{G}_C$ and $G_C$ are called 
    \term{directed cell subgraph} and \term{undirected cell subgraph}, respectively.
\end{definition}



\begin{wrapfigure}[20]{r}{0.33\textwidth}
    \centering
    \begin{subfigure}[b]{0.33\textwidth}
        \centering
        \begin{tikzpicture}[
            every node/.style={circle,thick,draw,inner sep=2.5},
            every edge/.style={draw,thick}]
            \node (0) at ( 0,1) {0};
            \node (1) at (-1,1) {1};
            \node (2) at (-1,0) {2};
            \node (3) at ( 0,0) {3};
            \node (4) at (0.85,-.2) {4};
            \node (6) at (0.85,1.2) {6};
            \node (5) at (1.5,0.5) {5};
            \path [->] (1) edge (2);
            \path [->] (0) edge (3);
            \path [->] (3) edge (0);
            \path [->] (2) edge (0);
            \path [->] (2) edge (3);
            \path [->] (6) edge (0);
            \path [->] (4) edge (3);
            \path [->] (5) edge (3);
            \path [->] (5) edge (0);
        \end{tikzpicture}
        \caption{
            Directed cell subgraph associated with a cell
            in the edge contraction subdivision 
            for the graph in \Cref{fig:G}.
        }
        \label{fig: running: directed cell subgraph}
    \end{subfigure}
    \vspace{0.5ex}
    
    \begin{subfigure}[b]{0.33\textwidth}
        \centering
        \begin{tikzpicture}[
            every node/.style={circle,thick,draw,inner sep=2.5},
            every edge/.style={draw,thick}]
            \node (0) at ( 0,1) {0};
            \node (1) at (-1,1) {1};
            \node (2) at (-1,0) {2};
            \node (3) at ( 0,0) {3};
            \node (4) at (0.85,-.2) {4};
            \node (6) at (0.85,1.2) {6};
            \node (5) at (1.5,0.5) {5};
            \path  (1) edge (2);
            \path  (0) edge [color=red,style={double}] (3);
            \path  (2) edge (0);
            \path  (2) edge (3);
            \path  (6) edge (0);
            \path  (4) edge (3);
            \path  (5) edge (3);
            \path (5) edge (0);
        \end{tikzpicture}
        \caption{The undirected cell subgraph corresponding to the above directed cell subgraph.}
        \label{fig: running: undirected cell subgraph}
    \end{subfigure}
    \caption{Cell subgraphs}
\end{wrapfigure}
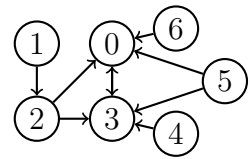
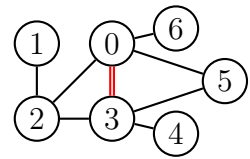
    
\Cref{fig: running: directed cell subgraph} shows an example of directed cell subgraph
associated with a cell in the edge contraction subdivision $\subdiv_{0,3}$
for the graph shown in \Cref{fig:G} induced by the contraction 
$G \mapsto G \sslash \{ \{0,3\} \}$.
\Cref{fig: running: undirected cell subgraph} shows the corresponding
undirected cell subgraph.
In the following, we explore the properties of such cell subgraphs.


Cycle spaces \cite{Gross2005GraphTheory} associated with undirected cell subgraphs 
play an important role in our understanding of the combinatorial structure of cells. 
Recall that the cycle space of a graph $G$ is a vector space 
spanned by all cycles in $G$ over the field $\Z_2$.
The dimension of this space, known as \emph{cyclomatic number}
(a.k.a. circuit rank), 
can be computed as the minimum number of edges one needs to remove to produce a spanning forest. 
With respect to a spanning tree $T$ of $G$, 
an edge not contained in $T$ is associated with a cycle,
the unique cycle formed by adding this edge to $T$.
This cycle is known as the \emph{fundamental cycle} associated with this edge.
A distinguished class of cycles, which we shall call ``balanced cycles'', 
will be the key concept in our discussion in this section.

\begin{definition}\label{balanced cycle}
    A cycle $O$ in an undirected graph $G$ is said to be \term{balanced}
    with respect to the edge contraction $G \mapsto G \sslash \{k_1,k_2\}$
    if $\edges(O) \setminus \{ \{ k_1,k_2 \} \}$ contains an even number edges.
    A subgraph of $G$ is \emph{balanced} 
    with respect to a edge contraction
    if all cycle in this subgraph are balanced
    with respect to this edge contraction.
    The \term{balanced circuit rank} of a $G$ is the maximum
    cyclomatic number (circuit rank) of its balanced subgraphs
    with respect to this edge contraction.
\end{definition}

Equivalently, a cycle $O$ of $G$ is balanced
with respect to the edge contraction $G \mapsto G \sslash \{k_1,k_2\}$
if and only if $O' = O \sslash \{ k_1, k_2 \}$ is a bipartite graph.
In this case, $O'$ will be a balanced bipartite graph.
When the edge contraction in question is clear from context,
we may simply say a cycle or subgraph is balanced
without explicit reference of the edge contraction.



\begin{theorem} \label{trm: properties of cell subgraphs}
    Let $C$ be a cell in the edge contraction subdivision $\subdiv_{k_1,k_2}$,
    then the corresponding cell subgraphs ${G}_C$ and $\digraph{G}_C$ have the following properties.
    \begin{enumerate} [label=(\roman*)] 
        \item \label{trm part: exactly one cycle} 
            $\digraph{G}_C$ contains exactly one directed cycle, 
            which is $k_1 \leftrightarrow k_2$.
            
        \item \label{trm part: C all nodes} 
            $\nodes({G}_C)=\nodes(\digraph{G}_C)=\nodes(G)$.
            
        \item
            $\edges(\digraph{G}_C)$ is closed under the map 
            \[
                \psi( (i,j) ) =
                \begin{cases}
                    (\sigma(i),\sigma(j)) &\text{if } (\sigma(i),\sigma(j)) \in \edges(G) \\
                    (i,j)   &\text{otherwise},
                \end{cases}
            \]
            where $\sigma$ is the transposition of the vertices given by 
            $k_1 \mapsto k_2$ and $k_2 \mapsto k_1$.
    
        \item \label{trm part: balanced cycle}  
            $G_C$ is balanced.
        \item \label{trm part: flip} 
            $G_C$ has a cycle basis with at most one odd cycle.
    \end{enumerate}
\end{theorem}
\begin{proof}
    Let $\hat F$ be a facet of the lifted polytope $\hat{\adjp}_G$ corresponding to a cell $C$,
    and let $\hat{\boldalpha} = (\boldalpha, 1) = (\alpha_1,\dots,\alpha_n,1)$ be an inner normal vector of $\hat F$.
    
    \begin{enumerate}[label=(\roman*),wide, labelwidth=!, itemindent=0em]
        \setlength{\itemsep}{1ex}%
    
        \item 
            By~\Cref{lem:special-edge}, $\digraph{G}_C$ contains a cycle $k_1 \leftrightarrow k_2$. 
            Suppose $\digraph{G}_C$ contains different directed cycle 
            $i_1 \leftrightarrow i_2 \leftrightarrow \cdots i_{m} \leftrightarrow i_{m+1}$
            with $i_{m+1} = i_1$, then by \Cref{Cor: h=0}, 
            \[ 
                0 = \biginner{ 
                    \left(\bolde_{i_r}-\bolde_{i_{r+1}}, 
                    \omega_{k_1,k_2}(\bolde_{i_r} - \bolde_{i_{r+1}})\right) 
                }{ 
                    \hat{\boldalpha}
                }, 
                \quad
                \text{for each } r=1,2,\dots,m.
            \]
            Summing up these $m$ equations, we obtain 
            \[
                0 = \biginner{
                    \left( 
                        \boldzero, 
                        \sum_{r=1}^m \omega_{k_1,k_2}(\bolde_{i_r}-\bolde_{i_{r+1}})
                    \right) 
                }{ 
                    \hat{\boldalpha} 
                }
                =
                \sum_{r=1}^m \omega_{k_1,k_2}(\bolde_{i_r}-\bolde_{i_{r+1}}). 
            \]
            This is a contradiction, since the sum of weights 
            $\sum_{r=1}^m \omega_{k_1,k_2}(\bolde_{i_r}-\bolde_{i_{r+1}})$ is strictly positive.
            Therefore, there can be no directed cycle other than $k_1 \leftrightarrow k_2$.
            
        \item   
            It follows from the \Cref{def: cell subgraphs} that $\nodes({G}_C) = \nodes(\digraph{G}_C)$ and $\nodes(\digraph{G}_C) \subseteq \nodes(G)$. 
            Suppose $m \in \nodes(G)$ and $m \not \in \nodes(\digraph{G}_C)$. 
            Then $\hat{F}$ has no vertex of the form $(\pm(\bolde_i - \bolde_m),\omega_{k_1,k_2}(\bolde_{i}-\bolde_{m}) )$ for any $i$, 
            Consequently, all points in $C$ have $0$ as $m$-th coordinate. 
            $C$ is therefore contained in the coordinate subspace orthogonal to $\bolde_m$
            and hence at most $(n-1)$-dimensional. 
            This contradicts with the assumption that $C$ is full dimensional. 
 
        \item   
            Suppose there is a vertex $i$ such that $\{i,k_1\}$ and $\{i,k_2 \}$ are both in $\edges(G)$. 
            Without loss of generality, we assume that $\digraph{G}_C$ contains a directed edge $( i, k_1 )$. 
            Then $(\bolde_{i}-\bolde_{k_1}) \in C$ and 
            \begin{align*}
                0 
                &= \biginner{ \left(\bolde_{i}-\bolde_{k_1}, \omega_{k_1,k_2}(\bolde_{i}-\bolde_{k_1})\right) }{\hat\boldalpha}\\
                &= \biginner{\bolde_{i}}{\boldalpha}-\alpha_{k_1}+\omega_{k_1,k_2}(\bolde_{i}-\bolde_{k_1}) \\
                &= \biginner{\bolde_{i}}{\boldalpha}-\alpha_{k_2}+\omega_{k_1,k_2}(\bolde_{i}-\bolde_{k_2}) \\
                &= \biginner{ \left(\bolde_{i}-\bolde_{k_2}, \omega_{k_1,k_2}(\bolde_{i}-\bolde_{k_2})\right) }{\hat\boldalpha},
            \end{align*}
            since $\alpha_{k_1}=\alpha_{k_2}$ (\Cref{Cor: h=0}).
            Hence, $(\bolde_{i}-\bolde_{k_2}) \in C$ and $\digraph{G}_C$ also contains the directed edge $( i, k_2 )$.
            
        \item   
            Suppose $G_C$ contains a cycle of length $m$
            containing edges $i_1 \leftrightarrow i_2 \leftrightarrow \cdots \leftrightarrow i_{m+1}$
            with $i_{m+1} = i_1$.
            Let $X$ be the largest subset of $C$ for which $G_{X}$ is this cycle,
            i.e., $X$ either contain both $\pm (\bolde_{k_1} - \bolde_{k_2})$ or neither of them.
            
            We first consider the case when $\pm (\bolde_{k_1} - \bolde_{k_2}) \not\in X$.
            In this case, 
            \begin{equation} \label{eq: X even cycle}
               X = \left\{ 
                    \lambda_r (\bolde_{i_r} - \bolde_{i_{r+1}}) \mid  
                    r=1, \dots, m, \; \bolde_{i_1}=\bolde_{i_{m+1}}, \; 
                    \lambda_r \in \{-1, 1\} 
                \right\},
            \end{equation}
            where $\lambda_1,\dots,\lambda_m \in \{ \pm 1 \}$ 
            indicate the directions of the edges in $\digraph{G}_{X}$
            (with $\lambda_r = +1$ if $(i_r,i_{r+1}) \in \edges(\digraph{G}_X)$ 
             and  $\lambda_r = -1$ if $(i_{r+1},i_r) \in \edges(\digraph{G}_X)$).
            \Cref{Cor: h=0} implies that 
            \begin{equation} \label{eq:1: m is even} 
                - \biginner{\bolde_{i_r} - \bolde_{i_{r+1}}}{\boldalpha} = 
                \lambda_r \quad\text{for } r=1, \dots, m.
            \end{equation}
            Summing up these $m$ equation, we obtain 
            \begin{equation}\label{eq:2: m is even}
                0=\sum_{r=1}^m \lambda_r.
            \end{equation}
            This implies that $m$ is even, and hence, $G_X$ is a balanced cycle.
            
            We shall now consider the case when $\pm (\bolde_{k_1} - \bolde_{k_2}) \in X$.
            Without loss of generality, we let $i_{1} = k_1$ and $i_{m} = k_2$.
            Recall that by \Cref{lem:special-edge}, a cell contains both $\pm (\bolde_{k_1}-\bolde_{k_2})$. 
            Thus
            \begin{align} \label{eq: X is odd cycle}
                X =
                \left\{ 
                    \lambda_r (\bolde_{i_r} - \bolde_{i_{r+1}}) \mid  r=1, \dots, m-1, 
                \right\} 
                \cup \left\{ \pm (\bolde_{k_1}-\bolde_{k_2}) \right\},
            \end{align} 
            where $\lambda_1,\dots,\lambda_m \in \{ \pm 1 \}$, again,
            indicate the directions of the edges in $\digraph{G}_{X}$.
            As in the previous case, by \Cref{Cor: h=0},   
            \begin{align} \label{eq:3: m is even}
                - \biginner{\bolde_{i_r}-\bolde_{i_{r+1}} }{ \boldalpha } &= \lambda_r, \quad\text{for } r=1, \dots, m-1,\\
                \biginner{\bolde_{i_{1}}-\bolde_{i_{m-1}} }{ \boldalpha } &= \biginner{\bolde_{k_1}-\bolde_{k_2} }{ \boldalpha }= 0. \nonumber
            \end{align}
            Summing up these $m$ equations, we obtain
            \begin{equation} \label{eq:4: m is even}
            0 =  \sum_{k=r}^{m-1} \lambda_r,
            \end{equation}
            which implies that $m-1$ must be even
            and hence $m$ itself must be odd.
            Recall that in this case, it is assumed that $k_1 \leftrightarrow k_2$ is in $G_X$, therefore $G_X$ is balanced.
            
        \item 
            We construct a suitable fundamental cycle basis of $G_C$ by 
            taking a spanning tree of $G_C$ that does not contain the edge $\{k_1,k_2\}$. 
            Then using part \ref{trm part: balanced cycle}, 
            we can see that such basis will contain at most one odd cycle,
            which is the fundamental cycle associated with the edge $\{k_1,k_2\}$.
    \end{enumerate}
\end{proof}

\begin{remark} \label{remark: balance}
    The above proof explains the motivation behind the definition of balanced cycles.
    If $X$ is the largest subset of a cell $C \in \subdiv_{k_1,k_2}$
    for which $G_X$ represents a balanced cycle,
    then $|\edges(\digraph{G}_X)|$ is even and 
    equations \eqref{eq:1: m is even} and \eqref{eq:4: m is even} imply that
    $\edges(\digraph{G}_X)$ consists of two disjoint subsets of equal size
    having opposite orientations in the cycle. The directed graph $\digraph{G}_X$ is ``balanced'' in this sense. The undirected graph $G_X$ is balanced in the sense that
    it is derived from a ``balanced'' directed graph $\digraph{G}_X$.
\end{remark}

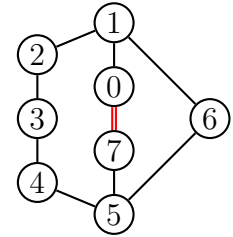
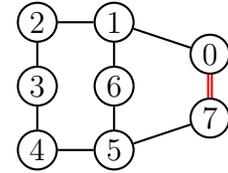
\begin{wrapfigure}{r}{0.3\textwidth}
    \centering
    \begin{subfigure}[t]{0.3\textwidth}
        \centering
        \begin{tikzpicture}[scale=0.85,
                every node/.style={circle,thick,draw,inner sep=1.8},
                every edge/.style={draw,thick}]
        \node (0) at (1.2, 0.5) {0};
        \node (1) at (1.2, 1.5) {1};
        \node (2) at (0, 1) {2};
        \node (3) at (0, 0) {3};
        \node (4) at (0, -1) {4};
        \node (5) at (1.2, -1.5) {5};
        \node (6) at (2.7, 0) {6};
        \node (7) at (1.2, -0.5) {7};
        \path (0) edge[color=red,style={double}] (7);
        \path (0) edge (1); 
        \path (1) edge (6);
        \path (6) edge (5); 
        \path (5) edge (7); 
        \path (1) edge (2); 
        \path (2) edge (3); 
        \path (3) edge (4); 
        \path (4) edge (5); 
        \end{tikzpicture}
        \caption{A graph $G$ created from gluing two odd cycles along three common edges.}
        \label{fig: cycles before flip}  
    \end{subfigure}
    \vspace{1ex}
    
    \begin{subfigure}[t]{0.3\textwidth}  
        \centering
        \begin{tikzpicture}[scale=0.85,
                every node/.style={circle,thick,draw,inner sep=1.8},
                every edge/.style={draw,thick}]
        \node (0) at (2.7, 0.5) {0};
        \node (1) at (1.2, 1) {1};
        \node (2) at (0, 1) {2};
        \node (3) at (0, 0) {3};
        \node (4) at (0, -1) {4};
        \node (5) at (1.2, -1) {5};
        \node (6) at (1.2, 0) {6};
        \node (7) at (2.7, -0.5) {7};
        \path (0) edge[color=red,style={double}] (7);
        \path (0) edge (1); 
        \path (1) edge (6);
        \path (6) edge (5); 
        \path (5) edge (7); 
        \path (1) edge (2); 
        \path (2) edge (3); 
        \path (3) edge (4); 
        \path (4) edge (5); 
        \end{tikzpicture}
    \caption{
        A different embedding of $G$
        shows it can also be considered as the result
        of gluing an even and an odd cycle along two edges.
    }
    \label{fig: cycles after flip}  
    \end{subfigure}
\caption{Sizes of cycles in cycle basis.}
\end{wrapfigure}
In general, the sizes of the chordless cycles in a cycle basis of a graph
are not uniquely defined.
\Cref{trm: properties of cell subgraphs} part \ref{trm part: flip}
will be used frequently in the following discussion to eliminate certain
ambiguities in cycle sizes.
It states that for an undirected cell subgraph,
there is always a choice of cycle basis
that contains at most one odd cycle and the rest are even cycles.
For example, the graph in \Cref{fig: cycles before flip} appears to be 
formed by gluing two odd cycles 
(a 5-cycle and a 7-cycle) along 3 common edges.
With a different embedding of the same graph in the plane,
it becomes apparent that it can also be formed by
gluing an even and an odd cycle along 2 common edges.
That is, it has a cycle basis that contains at most one odd cycle.
\Cref{trm: properties of cell subgraphs} part \ref{trm part: flip}
ensures that this is always possible for undirected cell subgraphs.

We now turn our attention to the combinatorial properties of cells and their subsets.
Recall that for a point configuration $X = \{ \boldx_1, \dots, \boldx_m \}$, 
its \emph{dimension}, denoted $\dim X$, 
is the dimension of smallest affine space that contains it.
It is said to be \emph{affinely dependent} if there are real coefficients
$\lambda_1,\dots,\lambda_m$ with $\lambda_1 + \cdots + \lambda_m = 0$
such that $\sum_{i=1}^m \lambda_i \boldx_i = \boldzero$.
Otherwise, it is \emph{affinely independent}.
$X$ is \emph{simplicial} if $|X| = \dim X + 1$,
and $X$ is a \emph{circuit} if it is affinely dependent
yet any of its proper subset is affinely independent.
The \emph{corank} of $X$ is the number $|X| - \dim X - 1$.
If $X$ has corank one, the coefficients $\lambda_1,\dots,\lambda_m$ in its affine dependence relation
 $\sum_{i=1}^m \lambda_i \boldx_i = \boldzero$
have a well defined sign pattern. That is, 
$\sigma^+ = |\{\lambda_i \mid \lambda_i > 0\}|$, 
$\sigma^- = |\{\lambda_i \mid \lambda_i < 0\}|$, and
$\sigma^0 = |\{\lambda_i \mid \lambda_i = 0\}|$
are uniquely defined, up to a permutation of $\sigma^+$ and $\sigma^-$,
and the tuple $(\sigma^+, \sigma^-, \sigma_0)$ is the \emph{signature} of $X$.

In the following, we establish the connections between the combinatorial properties
of a cell or a subset of a cell and the graph-theoretic properties of the corresponding subgraphs. 
This list of properties is reminiscent of a morphism between matroids.
This interpretation will be explored in \Cref{sec:matroid}.

\begin{theorem}\label{thm:matroid-morphism}
    Let $X$ be a nonempty subset of a cell $C \in \subdiv_{k_1,k_2}$ such that 
    $X$ either contains both $\pm (\bolde_{k_1} - \bolde_{k_2})$ or none of them, 
    then
    \begin{enumerate} [label=(\roman*)] 
        \item \label{trm part: tree - aff.indep}
            $X$ is affinely independent if and only if $G_X$ is a forest.
        
            
        \item \label{trm part: simple cycle - circuit}
            $X$ is a circuit if and only if $G_X$ is a chordless cycle.
            
        \item
            $\dim(X)=|\nodes(G_{X})| + |\edges(G_X) \cap \{ k_1, k_2 \}| - m - 1$
            where $m$ is the number of connected components in $G_X$.
            
        \item \label{trm part: corank=cycl.number}
            $\corank(X)$ equals the cyclomatic number of $G_X$.
    \end{enumerate}
\end{theorem}

\begin{proof}

    \begin{enumerate}[label=(\roman*),wide, labelwidth=!, itemindent=0em]
    
    \item 
    
        First, we show that if $G_X$ is not a forest
        ($G_X$ contains a cycle), then $X$ is affinely dependent.
        Since this cycle is also a subgraph of $G_C$, 
        by \Cref{trm: properties of cell subgraphs} \ref{trm part: balanced cycle}, it must be a balanced cycle.
        Let $X'$ be the largest subset of $X$ for which $G_X$ is this balanced cycle.
        It is sufficient to show that $X' \subset X$ is dependent.
        We consider two cases.

            
        In the cases where $\pm(\bolde_{k_1} - \bolde_{k_2}) \not\in X'$,
        $m = |X'|$ is even and it is the size of the cycle $G_{X'}$.
        The cycle thus consist of edges
        $i_1 \leftrightarrow \cdots \leftrightarrow i_{m} \leftrightarrow i_1$
        for some $i_j \in \{0,\dots,n\}$.
        As noted in \eqref{eq: X even cycle}, $X'$ can be expressed as
        \[ 
           \left\{ 
               \lambda_r (\bolde_{i_r} - \bolde_{i_{r+1}}) \mid  
               r=1, \dots, m, \; \bolde_{i_1}=\bolde_{i_{m+1}}, \; 
               \lambda_r \in \{-1, 1\} 
           \right\},
        \]
        where $\lambda_r$'s indicates the orientations of the edges in $\digraph{G}_{X'}$
        such that  $\sum_{r=1}^m \lambda_r=0$.
        We can see that
        \begin{equation} \label{eq: cycle dependence 1}
            \sum_{r=1}^m \lambda_r \lambda_r (\bolde_{i_r} - \bolde_{i_{r+1}}) = 
            \sum_{r=1}^m  (\bolde_{i_r} - \bolde_{i_{r+1}}) = \boldzero.
        \end{equation}
        Therefore $X'$ and hence $X$ itself are affinely dependent.
        
        In the cases where $\pm (\bolde_{k_1} - \bolde_{k_2}) \in X'$,
        $G_{X'}$ consists of edges
        $i_1 \leftrightarrow \cdots \leftrightarrow i_{m} \leftrightarrow i_1$ 
        with $i_1=k_1$ and $i_{m+1}=k_2$.
        Then $X'$ can be expressed as
        \begin{align*} 
            \left\{ 
                \lambda_r (\bolde_{i_r} - \bolde_{i_{r+1}}) \mid  
                r=1, \dots, m-1, 
                \lambda_r \in \{-1, 1\} \right
            \} 
            \cup 
            \left\{ 
                \pm (\bolde_{k_1}-\bolde_{k_2}) 
            \right\}
        \end{align*} 
        (See \eqref{eq: X is odd cycle}.)
        As in previous case, 
        $\sum_{k=1}^{m-1} \lambda_r=0$. 
        We can verify that  
        \begin{align} \label{eq: cycle dependence 2}
            \sum_{r=1}^{m} 
            \lambda_r \lambda_r (\bolde_{i_r} - \bolde_{i_{r+1}}) + 
            \left(-\frac{1}{2} \right)  (\bolde_{k_1}-\bolde_{k_2}) + 
            \frac{1}{2} (-\bolde_{k_1}+\bolde_{k_2}) 
            = \boldzero.
        \end{align}
        Therefore, $X'$ and hence $X$ are also dependent in this case.
        
        Conversely, suppose that $G_X$ is a forest, 
        we want to show $X$ is affinely independent. 
        It is sufficient to consider the case where $G_X$ is connected,
        i.e., $G_X$ is a tree.  
        If $\pm (\bolde_{k_1} - \bolde_{k_2}) \not\in X$,
        then by \Cref{trm: properties of cell subgraphs} \ref{trm part: exactly one cycle}, 
        $\digraph{G}_{X}$ is a directed acyclic graph.
        Through topological ordering, we can re-index the nodes
        so that $(i,j) \in \edges(\digraph{G}_{X})$ implies $i < j$ (or $j = 0$).
        In this case, 
        the matrix whose columns are $(\bolde_i-\bolde_j, 1)$ for $\bolde_i - \bolde_j \in X$
        is lower triangular with $\pm 1$ on the diagonal
        and therefore has rank $|X|$.
        Hence, $\dim(X) = |X| - 1$, and $X$ is affinely independent.
    
        On the other hand, if $\pm (\bolde_{k_1} - \bolde_{k_2}) \in X$, 
        we can re-index the nodes so that $k_1 = 0$, $k_2=1$, 
        and $(i,j) \in \edges(\digraph{G}_{X})$ implies $i < j$ (or $j = 0$).
        In this case,  the matrix whose columns are $(\bolde_i-\bolde_j-\bolde_1)$ 
        for $\bolde_i - \bolde_j \in X \setminus \{\bolde_1 \}$
        is upper triangular with diagonal entries $\pm 1$ or $-2$. 
        This matrix therefore has rank $|X|-1$.
        Hence, $\dim(X) = |X| - 1$, and $X$ is affinely independent.

    \item 
    
        We now establish the equivalence between
        $X$ being a circuit and $G_X$ being a chordless cycle.
        For one direction, 
        suppose $X$ is a circuit,
        then $X$ is dependent by definition.
        By part \ref{trm part: tree - aff.indep}, $G_X$ contains a cycle
        and the corresponding subset of points in $X$
        is dependent.
        However, the circuit $X$, being a minimal affinely dependent set,
        must be exactly this set.
        Therefore $G_X$ is exactly this cycle.
            
        Conversely, if $X$ is not a circuit,
        then either $X$ is affinely independent or 
        $X$ contains a proper affinely dependent subset $X'$.
        According to part \ref{trm part: tree - aff.indep}, 
        $G_X$ is either a forest or it contains a strictly smaller cycle,
        and thus $G_X$ is not a chordless cycle.
        
            
            
    \item 
        To compute $\dim(X)$, 
        it is sufficient to take a maximal affinely independent subset $X'$.
        Furthermore, we can assume $\pm (\bolde_{k_1} - \bolde_{k_2})$
        are contained in $X'$ if they are contained in $X$.
        By part \ref{trm part: tree - aff.indep},
        $G_{X'}$ is a spanning forest of $G_X$.
        Therefore 
        \[
            \dim(X) = \dim(X') =
            \begin{cases}
                |V(G_{X'})| - m     &\text{if } \pm (\bolde_{k_1} - \bolde_{k_2}) \in X' \\
                |V(G_{X'})| - m - 1 &\text{otherwise } 
            \end{cases}
        \]
        which produces the desired equality.
    
    \item 
We choose a spanning forest $G_{X'}$ of $G_X$
        which contains the edge $\{k_1, k_2\}$ if $\pm (\bolde_{k_1}, \bolde_{k_2}) \in X$. 
        This produces the maximally independent subset $X' \subseteq X$. 
        Thus, $\dim X' = \dim X$. 
        Let $\mu$ be the cyclomatic number of $G_X$.
        If $\pm (\bolde_{k_1}, \bolde_{k_2}) \not\in X$, then
        \begin{align*}
            \mu  
            &= |\edges(G_X)|-|\edges(G_{X'})| \\
            &= |X| - |X'| \\
            &= |X| - (\dim X' + 1) \\
            &= \corank(X).
        \end{align*}
        Similarly, if $\pm (\bolde_{k_1}, \bolde_{k_2}) \in X$, then
        \[
            \mu = |\edges(G_X)|-|\edges(G_{X'})| 
            = (|X|+1) - (|X'|+1) 
            = \corank(X).
        \]
    \end{enumerate}
\end{proof}

\begin{remark}\label{rmk: independent cycles}
    To highlight the necessity of
    the restriction in \Cref{thm:matroid-morphism} that
    the subset of a cell in question must contain both 
    $\pm (\bolde_{k_1} - \bolde_{k_2})$ or none of them ,
    we consider a subset $X$ of a cell that 
    include exactly one of the points $\pm (\bolde_{k_1} - \bolde_{k_2})$
    for which $G_X$ is a chordless cycle.
    We can verify that $X$ is affinely independent,
    in contrast with the result in 
     \Cref{thm:matroid-morphism} part \ref{trm part: tree - aff.indep}.
\end{remark}

For a corank-1 subset $X$ of a cell $C \in \subdiv_{k_1,k_2}$
which satisfies the condition that $X$ contains either 
both $\pm(\bolde_{k_1} - \bolde_{k_2})$ or none of them,
the graph $G_X$ contains a unique chordless cycle. By \Cref{trm: properties of cell subgraphs} \ref{trm part: balanced cycle}, this cycle must be balanced
in the sense that the set of corresponding directed edges
is partitioned into equal halves having opposite orientations
(\Cref{remark: balance}).
From this observation, we can derive signature of such  corank-1 subset.

\begin{corollary} \label{cor: signature for corank 1}
    Let $X$ be a corank-1 subset of $C \in \subdiv_{k_1,k_2}$
    such that
    $X$ contains either both $\pm(\bolde_{k_1} - \bolde_{k_2})$ or none of them,
    let $m$ be the circumference of $G_X$,
    then the signature of $X$ is $\left( \lceil \frac{m}{2} \rceil , \lceil \frac{m}{2} \rceil, |X|- 2 \lceil \frac{m}{2} \rceil \right) $.
\end{corollary}

\begin{proof}
    By \Cref{thm:matroid-morphism} \ref{trm part: corank=cycl.number}, 
    $X$ contains a unique chordless cycle, which is balanced. 
    Let $X' \subseteq X$ be the maximum subset for which $G_{X'}$ is this cycle,
    then $X'$ is a circuit,  by \Cref{thm:matroid-morphism} \ref{trm part: simple cycle - circuit}.
    As shown in \eqref{eq: cycle dependence 1} and \eqref{eq: cycle dependence 2}, exactly half of coefficients of affine dependence for $X'$ are positive, while another half are negative.  Moreover, and
    the points $\pm (\bolde_{k_1} - \bolde_{k_2})$, 
    if present, have coefficients with opposite signs. 
    Therefore, the signature of the circuit $X'$ is 
    $\left( \frac{|X'|}{2}, \frac{|X'|}{2}, 0 \right)$.
    The signature of $X$ is thus $\left( \frac{|X'|}{2}, \frac{|X'|}{2}, |X \setminus X'| \right)$.
    Let $m$ be the circumference of $G_X$, then
    \begin{align*}
        m= |\edges(G_{X'})|=\begin{cases}
        |X'|,    &\pm (\bolde_{k_1} - \bolde_{k_2}) \not\in X'\\
        |X'|-1,  &\pm (\bolde_{k_1} - \bolde_{k_2}) \in     X'.
        \end{cases}
    \end{align*}
    Thus $\frac{|X'|}{2} = \ceil*{ \frac{m}{2} }$,
    and the desired result follows immediately.
 \end{proof}

\begin{remark}
    It is worth interpreting the above result from the viewpoint of
    Radon's Theorem \cite{Radon1921Mengen},
    which states that a circuit $X \subset \R^n$
    can always be partitioned into two disjoint sets $X^+$ and $X^-$
    whose convex hulls have a nonempty intersection.
    The observation above shows that if $\pm (\bolde_{k_1} - \bolde_{k_2})$
    are both contained in a circuit $X$ in a cell $C \in \subdiv_{k_1,k_2}$,
    then the two points are separated by Radon's partition.
\end{remark}

\Cref{thm:matroid-morphism} lays out properties of subsets of a cell.
Viewing a cell $C \in \subdiv_{k_1,k_2}$ as a subset of itself,
which according to \Cref{lem:special-edge} satisfies the condition that $\pm(\bolde_{k_1}-\bolde_{k_2}) \in C$,
we derive the properties of a cell $C$.

\begin{corollary}\label{cor: cell invariants}
    Let $C \in \subdiv_{k_1,k_2}$ then
    \begin{enumerate} [label=(\roman*)] 
        \item
            $C$ is simplicial if and only $G_C$ is a spanning tree.
            
        \item
            $C$ is dependent if and only if $G_C$ contains a balanced cycle.
        
        \item
            $C$ is a circuit if and only if $G_C$ is a chordless cycle.
            
        \item \label{cor part: corank=cycl.number}
            The corank of $C$ equals the cyclomatic number of $G_C$.
    \end{enumerate}
\end{corollary}

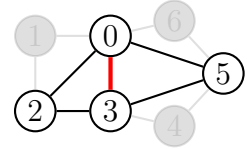
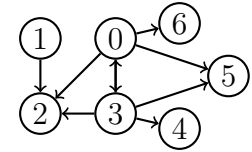
\begin{wrapfigure}{r}{0.36\textwidth}
    \centering
     \begin{subfigure}[b]{0.3\textwidth}
        \centering
        \begin{tikzpicture}[
            every node/.style={circle,thick,draw,inner sep=2},
            every edge/.style={draw,thick}]
            \node (0) at ( 0,1) {0};
            \node (1) at (-1,1) [color= gray,fill=gray,opacity=0.25] {1};
            \node (2) at (-1,0) {2};
            \node (3) at ( 0,0) {3};
            \node (4) at (0.85,-.2) [color=gray,fill=gray,opacity=0.25] {4};
            \node (6) at (0.85,1.2)[color=gray,fill=gray,opacity=0.25] {6};
            \node (5) at (1.5,0.5) {5};
            \path  (1) edge [color=gray,opacity=0.25] (2);
            \path  (1) edge [color=gray,opacity=0.25] (0);
            \path  (0) edge [color=red,style={ultra thick}] (3);
            \path  (2) edge (0);
            \path  (2) edge (3);
            \path  (6) edge [color=gray,opacity=0.25] (0);
            \path  (4) edge [color=gray,opacity=0.25] (3);
            \path  (4) edge [color=gray,opacity=0.25] (5);
            \path  (6) edge [color=gray,opacity=0.25] (5);
            \path  (5) edge (3);
            \path (5) edge (0);
        \end{tikzpicture}
        \caption{A balanced subgraph containing two cycles of the graph in \Cref{fig:G}}
        \label{fig: running: balanced subgraph}
    \end{subfigure}
    \vspace{2ex}
    
    \begin{subfigure}[b]{0.3\textwidth}
        \centering
        \begin{tikzpicture}[
            every node/.style={circle,thick,draw,inner sep=2},
            every edge/.style={draw,thick}]
            \node (0) at ( 0,1) {0};
            \node (1) at (-1,1) {1};
            \node (2) at (-1,0) {2};
            \node (3) at ( 0,0) {3};
            \node (4) at (0.85,-.2) {4};
            \node (6) at (0.85,1.2) {6};
            \node (5) at (1.5,0.5) {5};
            \path  (1) [->] edge (2);
            \path  (0) [->] edge (3);
            \path  (3) [->] edge (0);
            \path  (0) [->] edge (2);
            \path  (3) [->] edge (2);
            \path  (0) [->] edge (6);
            \path  (3) [->] edge (4);
            \path  (3) [->] edge (5);
            \path (0) [->] edge (5);
        \end{tikzpicture}
        \caption{The cell subgraph of a corank-2 cell.}
        \label{fig: running: cell subgraph}
    \end{subfigure}
    \caption{Maximum corank.}
\end{wrapfigure}
Combining part \ref{cor part: corank=cycl.number} of the above corollary 
and \Cref{trm: properties of cell subgraphs} \ref{trm part: balanced cycle},
we can conclude that the corank of a cell is bounded by the maximum cyclomatic number
of any balanced subgraph of $G$, which we call balanced circuit rank
(\Cref{balanced cycle}).
The graph $G$ shown in \Cref{fig:G}, for example,
has a balanced subgraph with cyclomatic number 2.
See \Cref{fig: running: balanced subgraph}.
In fact, it is the maximum cyclomatic number
any balanced subgraph of $G$ can have.
Therefore, we expect the maximum corank of cells in 
$\subdiv_{0,3}$ to be no more than 2.

Interestingly, this upper bound on corank is also attainable.
\Cref{fig: running: cell subgraph} shows the directed cell subgraph
associated with a cell in $\subdiv_{0,3}(G)$ derived from the
running example in \Cref{fig:G}.
This cell, containing 9 points in $\R^6$, is of corank 2,
which matches the maximum possible corank.
In the following, we show that this maximum possible corank,
given by the balanced circuit rank of the graph,
is always attainable.

\begin{theorem}\label{thm: max corank}
    The maximum corank of cells in $\subdiv_{k_1,k_2}(G)$ 
    is the balanced circuit rank of $G$. 
\end{theorem}
\begin{proof}
    Let $\mu$ be the balanced circuit rank of $G$ (\Cref{balanced cycle}). 
    Then by \Cref{trm: properties of cell subgraphs} \ref{trm part: balanced cycle}, 
    for any cell $C$, $G_C$ is a balanced subgraph of $G$. 
    Hence, $G_C$ has a cyclomatic number at most $\mu$,
    and by \Cref{thm:matroid-morphism} \ref{trm part: corank=cycl.number}, $\corank(C) \le \mu$.
     
    For the converse, we assume that there exists a balanced subgraph $G_S$ of $G$ 
    with a cyclomatic number strictly greater than $\mu$. 
    We want to show that there exists a cell $C$ with $\corank (C) > \mu$.
    Let $T$ be a spanning tree of $G$ containing the edge $\{k_1, k_2\}$.
    We will construct a subset $X$ of the vertices of $\adjp_G$ such that
    $G_X = T$, and points of $X$ form a basis of a cell in $\subdiv_{k_1,k_2}$.
    
    Taking $k_1$ to be the root, for any $i \in \nodes(G) \setminus \{k_1,k_2\}$,
    the spanning tree $T$ contains a unique path
    $k_1 \leftrightarrow i_1 \leftrightarrow \cdots \leftrightarrow i_{r-1} \leftrightarrow i_{r} = i$
    of length $r$. Thus the map
    \begin{equation*}
        \rho(i) = 
        \begin{cases}
        (-1)^{r-1} \, ( \bolde_{i_r} - \bolde_{i_{r-1}} ) & \text{if $i_1=k_2$},\\
        (-1)^{r}   \, ( \bolde_{i_r} - \bolde_{i_{r-1}} ) & \text{otherwise}
        \end{cases}
    \end{equation*}
    gives rise to a well-defined function $\rho : \nodes(G) \setminus \{k_1,k_2\} \to \adjp_G$.
    With this, we define
    \[
        X = \{ \rho(i) \mid i \in \nodes(G) \setminus \{k_1,k_2\} \}
        \;\cup\;
        \{ \pm (\bolde_{k_1} - \bolde_{k_2} \}.
    \]
    That is, along the the unique path in $T$ between $k_1$ and any node $i$ except $k_2$,
    $X$ contains the points representing directed edges along this path
    with alternating orientations.
    By this construction, $G_X = T$, $|X| = n+1$, and $X$ form a simplex in $\adjp_G$.
    Therefore there is a unique vector $\boldalpha \in \R^n$ such that
    \begin{align}\label{eq: alpha for spanning tree}
        \biginner{ \boldx }{ \boldalpha } + 1 &= 0
        \quad\text{for each } \boldx \in X \setminus \{ \pm (\bolde_{k_1} - \bolde_{k_2}) \}, \text{and} \\
        \biginner{ \boldx }{ \boldalpha } + 0 &= 0
        \quad\text{for } \boldx = \pm (\bolde_{k_1} - \bolde_{k_2}).\nonumber
    \end{align}
    In other words, each point in $X$ satisfies the equation
    $\inner{ \bullet }{ \boldalpha } + \omega_{k_1,k_2}(\bullet) = 0$.
    
    We can see for any $\boldx \in X \setminus \{ \pm (\bolde_{k_1} - \bolde_{k_2}) \}$,
    \[
        \biginner{ -\boldx }{ \boldalpha } + 1 =
        -\biginner{ \boldx }{ \boldalpha } + 1 =
        2 > 0.
    \]
    Therefore points in $-X$ satisfy the inequality
    $\inner{ \bullet }{ \boldalpha } + \omega_{k_1,k_2}(\bullet) \ge 0$.
    
    For any point $\bolde_i - \bolde_{i'}$ not in $X$ or $-X$,
    the corresponding edge $\{ i, i' \}$ is not in $T$.
    But $T$ is a spanning tree, so $i$ and $i'$ are connected through a unique path
    $i = i_1 \leftrightarrow \cdots \leftrightarrow i_m = i'$ 
    that is contained in $T$.
    By construction, $G_X = T$, thus there are $\lambda_1,\dots,\lambda_m \in \{ \pm 1 \}$, 
    representing the orientation of the directed edges,
    such that $\lambda_j (\bolde_{i_j} - \bolde_{i_{j+1}}) \in X$ for $j=1,\dots,m-1$.
    Moreover, $\lambda_1,\dots,\lambda_m$ carry alternating signs,
    except for those associated with $\pm (\bolde_{k_1} - \bolde_{k_2})$.
    We can verify that if $\{k_1, k_2\}$ is on this path, then 
    \begin{align*}
        \bolde_{i_1} - \bolde_{i_m} =
        \lambda_1 \lambda_1(\bolde_{i_1}-\bolde_{i_2})+ \dots +           \frac{1}{2}  (\bolde_{k_1}-\bolde_{k_2}) + &\left(-\frac{1}{2} \right) (-\bolde_{k_1}+\bolde_{k_2})+\dots  +\\
        &\lambda_{m-1}\lambda_{m-1}(\bolde_{i_{m-1}}-\bolde_{i_m}). 
    \end{align*}
    Otherwise, we have
    \[
        \bolde_{i_1}-\bolde_{i_m} =
        \lambda_1 \lambda_1(\bolde_{i_1}-\bolde_{i_2}) + \dots + \lambda_{m-1}\lambda_{m-1}(\bolde_{i_{m-1}}-\bolde_{i_m}). 
    \]
    In either case, by linearity and \eqref{eq: alpha for spanning tree},
    \begin{equation*}\label{eq:<extra edge, alpha>}
        \inner{ \bolde_{i} - \bolde_{i'} }{\boldalpha} = 
        \inner{ \bolde_{i_1} - \bolde_{i_m} }{\boldalpha} = 
        - \sum_{j \in J_{i,i'} } \lambda_j,
    \end{equation*}
    where $J_{i,i'}$ are the indices $j \in \{1,\dots,m\}$ 
    for which $\{ i_j, i_{j+1} \} \ne \{ k_1, k_2 \}$.
    By assumption, $\lambda_1,\dots,\lambda_m$ carry alternating sign pattern,
    thus the above sum must be in $\{ -1, 0, +1 \}$.
    Consequently,
    \begin{equation}\label{equ: equality in balanced cycle}
        \inner{ \bolde_{i} - \bolde_{i'} }{ \boldalpha } + 1 \ge 0.
    \end{equation}
    We can thus conclude that every vertex of $\adjp_G$ outside $\pm X$
    also satisfy the inequality
    $\inner{ \bullet }{ \boldalpha } + \omega_{k_1,k_2}(\bullet) \ge 0$.
    In other words, the equation $\inner{ \bullet }{ \boldalpha } + \omega_{k_1,k_2}(\bullet) = 0$
    defines a supporting hyperplane for $\hat{\adjp}_G$,
    and hence there is a unique cell $C \in \subdiv_{k_1,k_2}$ that contains $X$.
    We shall now show that this cell has corank strictly greater than $\mu$.
    
    
    We fix an edge $\{ i, i' \} \in \edges(G_S)$ that belongs to a cycle in $G_S$.
    Without loss of generality, we assume that this cycle is a 
    fundamental cycle formed by $\{i, i'\}$ with respect to $T$,
    i.e., it is formed by $\{ i, i' \}$ together with a unique path
    $i = i_1 \leftrightarrow \cdots \leftrightarrow i_m = i'$ in $T$.
    Let $\lambda_1,\dots,\lambda_m \in \{ \pm 1 \}$ be the coefficients
    representing orientations as described above.
    Since $G_S$ is assumed to be balanced, this cycle must be balanced.
    That is, either this cycle is an odd cycle which contains $\{ k_1, k_2 \}$
    or it is an even cycle which does not contain $\{ k_1, k_2 \}$.
    Thus $J_{i,i'} = \{ j \in \{1,\dots,m\} \mid \{ i_j, i_{j+1} \} \ne \{ k_1, k_2 \}$
    contains an odd number of indices.
    Following from \eqref{eq:<extra edge, alpha>},
    \[
        \inner{ \bolde_{i} - \bolde_{i'} }{\boldalpha} = 
        - \sum_{j \in J_{i,i'} } \lambda_j = \pm 1.
    \]
    Then either 
    \begin{align*}
        \inner{ \bolde_{i} - \bolde_{i'} }{\boldalpha} + 1 &= 0
        &&\text{or}&
        \inner{ \bolde_{i'} - \bolde_{i} }{\boldalpha} + 1 &= 0.
    \end{align*}
    Consequently, exactly one point in $\{ \pm (\bolde_{i} - \bolde_{i'} ) \}$
    is in the cell $C$ but not in $X$.
    
    Since the cyclomatic number of $G_S$ is strictly greater than $\mu$,
    there are strictly more than $\mu$ choices of 
    $\{ i, i' \} \in \edges(G_S) \setminus \edges(T)$.
    Therefore $| C \setminus X | > \mu$.
    Recall that $X$ itself is full-dimensional,
    thus $\corank(C) > \mu$.
\end{proof}

Applying the above theorem to trees and even cycles, 
which have no balanced cycles,
and odd cycles, which have balanced circuit rank of 1,
the following results can be derived immediately.

\begin{proposition}
    Let $G$ be a tree graph and $e = \{ k_1, k_2 \}$ be an edge in $G$. 
    Then any cell $C$ in the edge contraction subdivision $\subdiv_{k_1,k_2}$ is simplicial. 
\end{proposition}


\begin{proposition}
    Let $G$ be a cycle graph with an even number of vertices,
    let and $e = \{ k_1, k_2 \}$ be an edge in $G$. 
    Then the edge contraction subdivision $\subdiv_{k_1,k_2}$ is a triangulation. 
\end{proposition}

\begin{proposition}
    Let $G$ be a cycle graph with an odd number of vertices,
    let and $e = \{ k_1, k_2 \}$ be an edge in $G$. 
    Then every cell $C$ in the edge contraction subdivision $\subdiv_{k_1,k_2}$ is a circuit. 
\end{proposition}



We now procede to the volume computation problem.
As noted earlier, the normalized volume of $\adjp_G$ is an important property
that has found applications in the study of Kuramoto networks.
Since $\subdiv_{k_1,k_2}$ form a subdivision of the polytope $\adjp_G$,
the normalized volume of $\adjp_G$ can be computed as the sum of the
normalized volume of the cells in this subdivision.
While the problem of computing the normalized volume of an arbitrary cell
may be difficult, we shall show, in the following, that
the normalized volume of cells of small coranks can be computed directly.

\begin{theorem}\label{thm: cell volume}
    Let $\subdiv_{k_1,k_2}$ be an edge contraction subdivision (as in \Cref{def:edge-subdiv}),
    and let $C \in \subdiv_{k_1,k_2}$ be a cell.
    \begin{enumerate} [label=(\roman*)] 
        \item If $C$ is of corank-0, then \[ \nvol(C) = 2. \]
        
        \item
            If $C$ is of corank 1, then \[ \nvol(C) = m \]
            where $m$ is the circumference of $G_C$.
            
        \item
            If $C$ is of corank 2, and hence $G_C$ contains two chordless cycles 
            of size $m_1$ and $m_2$ respectively with at least one of them even, 
            then 
            \[
                \nvol(C) = \frac{m_1 m_2}{2} - 2 \gamma \delta,
            \]
            where $\gamma$ and $\delta$ are the numbers of directed edges shared by 
            the two chordless cycles in the two directions respectively.
    \end{enumerate}
\end{theorem}

\begin{proof}
\begin{enumerate}[label=(\roman*),wide, labelwidth=!, itemindent=0em]
    \item
        To simplify the notation, we can assume $k_1 = 0$ and $k_2 = n$.
        Suppose $C$ is simplicial, then by \Cref{cor: C simplex -> F simplex}, 
        $C$ is associated to a simplicial facet $F'$ of the polytope $\adjp_{G'}$,
        where $G' = G \sslash \{ 0, n \}$.
        That is, $F' = \conv \{ \bolda_1, \dots, \bolda_{n-1} \}$ for some
        $\bolda_1,\dots,\bolda_{n-1} \in \phi(\edges(G')) \subset \R^{n-1}$.
        Moreover, the projection of vertices in $C$ into $\R^{n-1}$
        are exactly the points $\boldzero,\bolda_1,\dots,\bolda_{n-1}$.
        Therefore,
        \[
            \nvol(C) =
            \left|
            \det
            \begin{bmatrix}
                \bolda^\top_1     & 1               \\
                \vdots            & \vdots          \\
                \bolda^\top_{n-1} & 1               \\
                                  & 2 \bolde_n^\top \\
            \end{bmatrix}
            \right|
            =
            2
            \left|
            \det
            \begin{bmatrix}
                \bolda_1^\top     \\
                \vdots            \\
                \bolda_{n-1}^\top \\
            \end{bmatrix}
            \right|
            ,
        \]
        which is precisely $2 \cdot \nvol_{k-1}(F')$.
        As shown in~\cite{Chen2019Directed,ChenDavisMehta2018Counting},
        a simplicial facet of an adjacency polytope has normalized volume 1,
        which produces the desired equality.
        
    \item
        Suppose $C$ is of corank 1, 
        by \Cref{thm:matroid-morphism} parts \ref{trm part: simple cycle - circuit} and \ref{trm part: corank=cycl.number}, 
        there is a unique circuit $X \subseteq C$ which contains either 
        both $\pm (\bolde_{k_1}-\bolde_{k_2})$ or none them 
        and which corresponds to a unique chordless cycle $G_X$ in $G_C$.
        Let $m$ be the size of this cycle which will also be the circumference of $G_C$.
        Let $C^{+}$ be the subset of $C$ contributing to the $\sigma_+$ in the signature of $C$,
        then as stated in \Cref{cor: signature for corank 1}, 
        $C^+ \subset X$ and $|C^+| = \lceil \frac{m}{2} \rceil$.
        We consider the triangulation \cite[Proposition 1.2]{gelfand_discriminants_1994} of $C$ given by
        \[
            D = \{ \, C \setminus \{\boldc\} \;\mid\; \boldc \in C^+ \,\},
        \]
        which contains $\ceil{ \frac{m}{2} }$ elements that will be referred to as subcells.
        
        If $\pm (\bolde_{k_1} - \bolde_{k_2}) \not \in X$ then $m$ is even
        and every $C' \in D$ must contain $\pm (\bolde_{k_1} - \bolde_{k_2})$.
        By the previous part, $\nvol(C') = 2$.
        Thus
        \[
            \nvol(C) = \sum_{C' \in D} \nvol(C') = \ceil*{ \frac{m}{2} } \cdot 2 = \frac{m}{2} \cdot 2 = m.
        \]
        
        On the other hand, if $\pm (\bolde_{k_1} - \bolde_{k_2}) \in X$,
        then $m$ is odd, and by \Cref{cor: signature for corank 1}, 
        only one of the points $\pm (\bolde_{k_1}-\bolde_{k_2})$ is in $C^+$. 
        Without loss of generality, we can assume $\bolde_{k_1}-\bolde_{k_2} \in C^+$
        In this case,
        \[
            D = 
            \{ \, C \setminus \{\boldc\} \;\mid\; \boldc \in C^+ \setminus \{ \bolde_{k_1} - \bolde_{k_2} \} \,\}
            \; \cup \;
            \{ \, C \setminus \{ \bolde_{k_1} - \bolde_{k_2} \}  \, \}.
        \]
        As in the above case, $\nvol( C \setminus \{ \boldc \} ) = 2$ 
        for $\boldc \ne \bolde_{k_1} - \bolde_{k_2}$ by the previous part.
        On the other hand, $\nvol( C \setminus \{ \bolde_{k_1} - \bolde_{k_2} \} ) = 1$.
        Therefore
        \[
            \nvol(C) = 
            \sum_{\boldc \in C^+ \setminus \{ \bolde_{k_1} - \bolde_{k_2} \} } \nvol( C \setminus \{ \boldc \} )  +
            \nvol( C \setminus \{ \bolde_{k_1} - \bolde_{k_2} \} ) =
            \frac{m-1}{2} \cdot 2 + 1 =
            m
        \]
        
    \item 
    Suppose $C$ is of corank 2, then by \Cref{thm:matroid-morphism} \ref{trm part: corank=cycl.number} 
    and \Cref{trm: properties of cell subgraphs} \ref{trm part: flip}, 
    $G_C$ contains two chordless cycles at least one of which is even. 
    Let $C_1$ and $C_2$ be the maximum subsets of $C$ representing the edges in these two cycles,
    and let $m_1 = |\edges(G_{C_1})|$ and $m_2 = |\edges(G_{C_2})|$ 
    with $\pm (\bolde_{k_1} - \bolde_{k_2}) \not\in C_2$ and $m_2$ being even.
    We will construct a regular subdivision for $C$ itself.
    Fixing any point $\bolda \in C_2 \setminus C_1$,
    We consider the lifting function $\omega_{C,\bolda} : C \to \Z$ given by
    \[
        \omega_{C,\bolda} (\boldc) =
        \begin{cases}
            1 &\text{if } \boldc = \bolda, \\
            0 &\text{otherwise},
        \end{cases}
    \]
    which induces a proper subdivision of the cell $C$ itself.
    Cells in this subdivision, which correspond to lower facets of
    $\conv \{ (\boldc, \omega_{C,\bolda} ) \mid \boldc \in C \}$,
    will be referred to as the subcells of $C$.
    Since $C$ is of corank 2, proper subcells of $C$ are either of corank 1 or 0.
    We shall enumerate all the subcells in this subdivision. 
    
    For convenience, a point $\boldb$ is referred to as being positively oriented
    if it represents an edge which has the same orientation as $\bolda$
    in some common (undirected) cycle of $G_C$ that contains both. 
    We first investigate all possible corank-1 subcells $C'$ of $C$,
    which must be of the form $C' = C \setminus \{ \boldb \}$ for some $\boldb \in C$. 
    Let $(\boldalpha,h) \in \R^{n+1}$ be the unique vector such that 
    $C'$ consists of all points $\boldc \in C$ for which 
    $\biginner{ \boldc }{ \boldalpha } + \omega_{C,\bolda}(\boldc) = h$.
    That is, $\biginner{ \bullet }{ (\boldalpha,1) } = h$ 
    defines the supporting hyperplane for the lower facet of 
    $\conv\{ (\boldc, \omega_{C,\bolda}(\boldc)) \mid \boldc \in C \}$
    whose projection is $C'$.
        
    We will first show it is necessary that $\boldb \in C_2 \setminus C_1$. 
    If $ \boldb \not \in C_1 \cup C_2$,
    then $G_{C'}$ contains two chordless cycles,
    and by \Cref{thm:matroid-morphism} \ref{trm part: corank=cycl.number}, 
    $C_1 \cup C_2$ and consequently $C'$ are of corank-2, 
    which contradicts to our assumption.
    If, on the other hand, $\boldb \in C_1$,
    then we can express $\boldb$ as an affine combination
    \[
        \boldb =
        \lambda_1 \boldc_1 + \cdots + \lambda_{m_1} \boldc_{|C_1|-1},
    \]
    where $\sum_i \lambda_i = 1$ and 
    $\{ \boldc_1, \dots, \boldc_{|C_1|-1} \} = C_1 \setminus \{ \boldb \} \subset C'$.
    Then
    \begin{align*}
        \biginner{ \boldb }{ \boldalpha } 
        &=       \lambda_1         \biginner{ \boldc_1         }{ \boldalpha } + 
        \cdots + \lambda_{|C_1|-1} \biginner{ \boldc_{|C_1|-1} }{ \boldalpha }
        = (\lambda_1 + \cdots + \lambda_{|C_1|-1})\, h 
        = h,
    \end{align*}
    which implies that $\boldb \in C'$ and contradicts with our assumption
    that $C' = C \setminus \{ \boldb \}$. 
    The only possibility left is $\boldb \in C_2 \setminus C_1$.
        
    We will now show that, under this assumption, 
    $C' = C \setminus \{ \boldb \}$ is a corank-1 subcell of $C$ 
    if and only if $\boldb$ is positively oriented. 
    This is equivalent to the statement that
    $(\boldb,0)$ is the only point in 
    $\conv \{ (\boldc,\omega_{C,\bolda}(\boldc) \mid \boldc \in C\}$ 
    that is strictly above the supporting hyperplane
    defined by $\biginner{ \bullet }{ (\boldalpha,1) } = h$
    if and only if $\boldb$ is positively oriented.
    First, we note that $G_{C'}$ is a connected spanning subgraph of $G$, 
    and $\digraph{G}_{C'}$ contains both directed edges $(k_1, k_2)$ and $(k_2, k_1)$. 
    Thus, we can conclude that
    $h = \inner{ \pm(\bolde_{k_1} - \bolde_{k_2}) }{ \boldalpha }$ must be zero.
    That is, for all $\boldc \in C'$,
    \[
        \biginner{ \boldc }{ \boldalpha } =
        \begin{cases}
            -1 &\text{if } \boldc = \bolda, \\
            0 &\text{otherwise}.
        \end{cases} 
    \]
    Since $C'$ contains $C_2 \setminus \{ \boldb \}$,
    and $G_{C_2}$ is an even cycle,
    there are $\lambda_{\boldb}, \lambda_1, \lambda_2, \dots, \lambda_{m_2-1} \in \{-1,1\}$
    which sum to exactly 0 such that
    \[
        -\lambda_{\boldb} \boldb = 
        \lambda_1 \boldc_1 + \lambda_2 \boldc_2 + \cdots + \lambda_{m_2-1} \boldc_{m_2-1},
    \]
    where $\boldc_1,\boldc_2,\dots,\boldc_{m_2-1}$ represent the rest of the edges
    on the cycle $G_{C_2}$.
    The choices of $\{ \lambda_i \}_{i=1}^{m_2-1}$ are unique up to a uniform change of signs.
    Without loss of generality, we can choose $\lambda_i = 1$ if $\boldc_i = \bolda$
    in the case $\boldb \ne \bolda$.
    Then
    \begin{align*}
        \biginner{ \boldb }{ \boldalpha } + \omega_{C,\bolda}(\boldb) 
        = - \lambda_{\boldb}
            \sum_{i=1}^{m_2-1} \lambda_i \inner{\boldc_i}{\boldalpha} 
            + \omega_{C,\bolda}(\boldb)
        =
        \begin{cases}
            \lambda_{\boldb} &\text{if } \boldb \neq\bolda, \\
            1                &\text{if $\boldb=\bolda$}.
        \end{cases}
    \end{align*}
    Consequently, the condition 
    $\inner{ (\boldb,\omega_{C,\bolda}(\boldb) }{ (\boldalpha,1) } > 0$
    is equivalent to the condition that
    either $\boldb = \bolda$ or $\lambda_{\boldb} > 0$.
    These two cases can be combined into the condition that
    $\boldb$ is positively oriented
    (including the case $\boldb = \bolda$).
    Therefore $C' = C \setminus \{ \boldb \}$ is a corank-1 subcell of $C$
    in the regular subdivision induced by $\omega_{C,\bolda}$
    if and only if $\boldb \in C_2 \setminus C_1$ is positively oriented.
    
    The number of choices for $\boldb$ that satisfy this description is
    $\frac{ m_2 }{2} - \gamma$,
    where $\gamma$ is the number of positively oriented points in $C_1 \cap C_2$.
    By the previous part,
    the normalized volume of each subcell is $m_1$.
    The contribution of corank-1 subcells to $\nvol(C)$ is therefore
    \begin{equation}
        m_1 \left(
            \frac{ m_2 }{2} - \gamma
        \right).
        \label{equ: corank-1 in corank-2}
    \end{equation}

    It is also possible to have corank-0 subcells in the regular subdivision
    induced by $\omega_{C,\bolda}$.
    Suppose there is a corank-0 subcell $T$.
    We can see that $\bolda$ must be in $T$
    since any subset of $C \setminus \{ \bolda \}$ must be contained in a
    corank-1 subcell described above.
    As in the previous case, we still let $(\boldalpha,h) \in \R^{n+1}$ 
    be the unique vector such that
    $T$ consists of all points $\boldc \in C$ such that
    $\biginner{ \boldc }{ \boldalpha } + \omega_{C,\bolda} = h$.
    In particular, $(\boldalpha,h)$ satisfies equations
    \begin{align*}
        \biginner{ \boldc }{ \boldalpha } =
        \begin{cases}
            h - 1 &\text{if } \boldc = \bolda, \\
            h     &\text{otherwise},
        \end{cases}
    \end{align*}
    for all $\boldc \in T$.
    By definition, $\dim T = n$ and $|T| = n+1 = |C| - 2$.
    In particular, there are distinct $\boldb_1,\boldb_2 \in C$ such that
    $T = C \setminus \{ \boldb_1, \boldb_2 \}$. 
    Moreover, $T$ must be affinely independent,
    so $T$ cannot contain dependent sets $C_1$ or $C_2$.
    Consequently, $C_1 \setminus T, C_2 \setminus T \ne \varnothing$.
    Without loss of generality, assume 
    $\boldb_1 \in C_1 \setminus T$ and
    $\boldb_2 \in C_2 \setminus T$.
    We will show that it is necessary that
    $\boldb_1 \in C_1 \setminus C_2$ and 
    $\boldb_2 \in C_1 \cap C_2$.
        
    First, suppose $\boldb_1,\boldb_2 \in C_1 \cap C_2$,
    then $T$ contains $\pm (\bolde_{k_1} - \bolde_{k_2})$,
    and $G_T$ still contains a cycle. 
    By \Cref{thm:matroid-morphism} \ref{trm part: tree - aff.indep}, $T$ must be affinely dependent, 
    which contradicts with our assumption.
    Now suppose $\boldb_1 \in C_1$ and $\boldb_2 \in C_2 \setminus C_1$,
    then 
    \[
        \boldb_1 = 
        \begin{cases}
            \lambda_1 \boldc_1 + \cdots \lambda_{m_1} \boldc_{m_1}
            &\text{if } \{ \pm (\bolde_{k_1} - \bolde_{k_2}) \} \subset C_1,
            \\
            \lambda_2 \boldc_2 + \cdots \lambda_{m_1} \boldc_{m_1}
            &\text{otherwise},
        \end{cases}
    \]
    where $\lambda_k$'s sum to 1,
    and $\boldc_k$'s represent the edges in $G_{C_1 \setminus \{\boldb_1\}}$.
    Then, in both cases,
    \[
        \biginner{ \boldb_1 }{ \boldalpha } = h.
    \]
    Thus $\boldb_1 \in T$ by definition, which also contradict with our assumptions.
    We can see that the only possibilities left are those with
    $\boldb_1 \in C_1 \setminus C_2$ and $\boldb_2 \in C_1 \cap C_2$.
    In particular, this implies that there would be no corank-0 subcells
    if $C_1 \cap C_2 = \varnothing$.
        
    We shall now show that under this assumption
    ($\boldb_1 \in C_1 \setminus C_2$ and $\boldb_2 \in C_1 \cap C_2$),
    $T = C \setminus \{ \boldb_1, \boldb_2 \}$
    is a subcell of corank 0 if and only both points are positively oriented.
    First, note that $G_T$ contains no cycles, 
    so by \Cref{thm:matroid-morphism} \ref{trm part: tree - aff.indep}, $T$ is affinely independent.
    Moreover, $|T| = |C| - 2 = n+1$, thus $T$ is full-dimensional.
    Since both edges represented by $\bolda$ and $\boldb_2$ 
    are contained in the chordless cycle $G_{C_2}$, 
    we can express $\boldb_2$ as an affine combination of the form
    \[
        \boldb_2 = \lambda \bolda + \lambda_2 \boldc_2 + \cdots \lambda_{m_2-1} \boldc_{m_2-1},
    \]
    where $\lambda + \lambda_2 + \cdots + \lambda_{m_2-1} = 1$ 
    and $\bolda,\boldc_2,\dots,\boldc_{m_2-1}$
    represent the rest of the edges of the cycle $G_{C_2}$.
    Then, as in the case above,
    \begin{align*}
        \biginner{ \boldb_2 }{ \boldalpha }
        &=
        \lambda   \inner{\bolda}{\boldalpha} + 
        \lambda_2 \inner{\boldc_2}{\boldalpha} + 
        \cdots + 
        \lambda_{m_2-1} \inner{\boldc_{m_2-1}}{\boldalpha} 
        \\
        &=
        \lambda (h - 1) + \lambda_2 \cdot h + \cdots + \lambda_{m_2-1} \cdot h
        \\
        &=
        (\lambda + \lambda_2 + \cdots + \lambda_{m_2-1}) \cdot h - \lambda
        \\
        &= h - \lambda,
    \end{align*}
    which is strictly greater than $h$ 
    if and only if $\lambda < 0$.
    This condition is equivalent to the condition that
    $\boldb_2$ is positively oriented.
    
    Since the cycles $G_{C_1}$ and $G_{C_2}$ share at least one edge
    represented by $\boldb_2$, 
    we can see that edges represented by $\boldb_1$ and $\bolda$
    are also contained in a common cycle.
    By the same argument, we can also conclude that
    $\inner{ \boldb_1 }{ \boldalpha } + \omega_{C,\bolda}(\boldb_1) > h$
    if and only $\boldb_1$ is positively oriented.
    
    That is, both $(\boldb_1,0)$ and $(\boldb_2,0)$ 
    will lie above the hyperplane  defined by 
    $\inner{ \bullet }{ (\boldalpha,1) } = h$,
    if and only if both $\boldb_1$ and $\boldb_2$ are positively oriented. Under this condition, $T = C \setminus \{ \boldb_1, \boldb_2 \}$ 
    is a corank-0 subcell of $C$ induced by the lifting $\omega_{C,\bolda}$.
        
    In terms of their contributions to the normalized volume of $C$,
    corank-0 subcells, if exist, come in two different types.
    
    The first type of corank-0 subcells are subcells 
    $T = C \setminus \{ \boldb_1, \boldb_2 \}$ with positively oriented
    $\boldb_1 \in C_1 \setminus C_2$ and $\boldb_2 \in C_1 \cap C_2$
    for which  $\pm (\bolde_{k_1} - \bolde_{k_2}) \not\subset T$.
    As shown above, $\nvol(T) = 1$.
    Since either $\pm (\bolde_{k_1} - \bolde_{k_2}) $ are not in $C_2$,
    it is necessary to choose  $\boldb_1 \in \{ \pm (\bolde_{k_1} - \bolde_{k_2}) \}$
    and the choice is unique if $\{ \pm (\bolde_{k_1} - \bolde_{k_2}) \} \subset C_1$.
    Under this assumption,
    the number of distinct choices of $\boldb_2$ is $\gamma$,
    the number of positively oriented points in $C_1 \cap C_2$.
    Therefore, the total number of corank-0 subcells of this type is
    \[
        \begin{cases}
            \gamma &\text{if } \{ \pm (\bolde_{k_1} - \bolde_{k_2}) \} \subset C_1,
            \\
            0 &\text{otherwise}.
        \end{cases}
    \]
    
    The second type of corank-0 subcells are of the form
    $T = C \setminus \{ \boldb_1, \boldb_2 \}$ with positively oriented
    $\boldb_1 \in C_1 \setminus C_2$ and $\boldb_2 \in C_1 \cap C_2$
    for which  $\pm (\bolde_{k_1} - \bolde_{k_2}) \subset T$.
    By the first part of this theorem, $\nvol(T) = 2$.
    The number of distinct choices for $\boldb_2$ is $\gamma$.
    Point $\boldb_1$ must be chosen from the set of positively oriented points
    in $C_1 \setminus C_2 \setminus \{ \pm (\bolde_{k_1} - \bolde_{k_2}) \}$.
    So the total number of such combinations is given by
    \[
        \begin{cases}
            \left( \ceil*{\frac{ m_1 + m_2 - 2(\gamma + \delta) }{2}} - \ceil*{\frac{ m_2 }{2}} + \gamma - 1 \right) \cdot \gamma
            &\text{if } \{ \pm (\bolde_{k_1} - \bolde_{k_2}) \} \subset C_1
            \\
            \left( \ceil*{\frac{ m_1 + m_2 - 2(\gamma + \delta) }{2}} - \ceil*{\frac{ m_2 }{2}} + \gamma \right) \cdot \gamma
            &\text{otherwise}.
        \end{cases}
    \]
    If $\{ \pm (\bolde_{k_1} - \bolde_{k_2}) \} \subset C_1$,
    then by \Cref{trm: properties of cell subgraphs} \ref{trm part: balanced cycle},
    $m_1$ is odd, and
    \begin{align*}
        \nvol(C) &=
        m_1 \left( \frac{m_2}{2} - \gamma \right) +
        1 \cdot \gamma +
        2 \cdot \left( \ceil*{\frac{ m_1 + m_2 - 2(\gamma + \delta) }{2}} - \ceil*{\frac{ m_2 }{2}} + \gamma - 1 \right) \cdot \gamma
        \\
        &= \frac{m_1 m_2}{2} - 2 \gamma \delta.
    \end{align*}
    Similarly, if $\{ \pm (\bolde_{k_1} - \bolde_{k_2}) \} \not\subset C_1$,
    $m_1$ is even, and
    \begin{align*}
        \nvol(C) &=
        m_1 \left( \frac{m_2}{2} - \gamma \right) +
        2 \cdot \left( \ceil*{\frac{ m_1 + m_2 - 2(\gamma + \delta) }{2}} - \ceil*{\frac{ m_2 }{2}} + \gamma \right) \cdot \gamma
        \\
        &= \frac{m_1 m_2}{2} - 2 \gamma \delta.
    \end{align*}
        
    \end{enumerate}
\end{proof}

\section{The matroid point of view}\label{sec:matroid}

\Cref{thm:matroid-morphism} highlights 
the connection between subsets of cells in $\subdiv_{k_1,k_2}$ and  subgraphs of $G$.
Through this connection, independent subsets correspond to forests,
dependent subsets correspond to cyclic graphs,
and circuit correspond to chordless cycles.
A precise characterization emerges from the point of view of matroid theory.

Recall that a matroid consists of a finite \emph{ground set},
together with a family of its subsets, known as \emph{independent sets}
that satisfies the conditions
\begin{enumerate}
    \item $\varnothing$ is an independent set;
    \item a subset of a independent set is independent;
    \item for two independent sets $A$ and $B$ with $|A| < |B|$,
        $A \cup \{ b \}$ is independent for some $b \in B$.
\end{enumerate}
Subsets of the ground set that are not independent
are called \emph{dependent sets}.
A maximally independent set is known as a \emph{basis}
while a minimally dependent set is known as a \emph{circuit}.
Two matroid structures appear naturally in this context,
and \cref{thm:matroid-morphism} shows that 
they are essentially the mirror images of one another.

Fixing a cell $C \in \subdiv_{k_1,k_2}$, consider the ground set,
\[
    P_C = 
    \{ \{ \boldc \} \mid \boldc \ne \pm (\bolde_{k_1} - \bolde_{k_2}) \}
    \cup
    \{ \{ \pm (\bolde_{k_1} - \bolde_{k_2} \} \}.
\]
Based on this ground set, we define the matroid with the family of
independent sets consiting of any set $X \subset P_C$ for which
$\bigcup_{\boldx \in X} \boldx$ is affinely independent,
including the empty set.

On the other hand, building on the ground set $E_C = \edges(G_C)$,
we can also construct a matroid where the independent elements are
subsets of edges, including the empty set, that does not contain a cycle.
In this setup, a basis is a spanning tree,
and a circuit is a chordless cycle.

\Cref{thm:matroid-morphism} shows the parallel nature of 
these two matroid structures in the context of this paper,
and it can be interpreted as follows.

\begin{proposition}
    Fixing a cell $C \in \subdiv_{k_1,k_2}(G)$,
    let $P_C$ and $E_C$ be the two matroid structures defined above.
    Then the function $f : 2^{P_C} \to 2^{E_C}$ given by $f(X) = G_{X}$
    has the properties that
    \begin{enumerate}[label=(\roman*)] 
        \item $f(X)$ is a basis if and only if $X$ is a basis;
        \item $f(X)$ is a circuit if and only if $X$ is a circuit;
        \item $f(X)$ is dependent if and only $X$ is dependent;
        \item $\rank(f(X)) = \rank(X)$;
    \end{enumerate}
\end{proposition}


\section{Examples}\label{sec:examples}

In this section, we show a few examples of edge contraction subdivisions
for adjacency polytopes.
For ease of visualization, we show directed cell subgraphs
whose edges are in one-to-one correspondence with points of cells.

\subsection{A 4-cycle}

\begin{wrapfigure}{r}{0.25\textwidth}
        \centering
        \begin{tikzpicture}[
            every node/.style={circle,thick,draw},
            every edge/.style={draw,thick}]
            \node (0) at ( 0.5  , .7) {0};
            \node (1) at (-0.3,  0) {1};
            \node (3) at ( 1.3,  0) {3};
            \node (2) at ( 0.5,-.7) {2};
            \path (0) edge (1);
            \path (0) edge[color=red,style={double}] (3);
            \path (1) edge (2);
            \path (2) edge (3);
        \end{tikzpicture}
        \caption{$C_4$ is a cycle of 4 nodes}
        \label{fig:C4}
    
\end{wrapfigure}
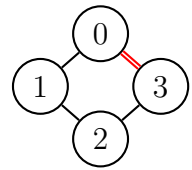

\Cref{fig:C4} shows a cycle of 4 nodes $C_4$.
Fixing $\{0,3\}$ to be the contraction edge,
the induced edge contraction subdivision $\subdiv_{0,3}(C_4)$
contains 6 cells which are in one-to-one correspondence
with the 6 facets of the adjacency polytope for $C_3 = C_4 \sslash \{ 0, 3 \}$
(cf.~\cite[Figure 2]{Chen2019Directed}). 
Since the only cycle in $C_4$ is of length 4, 
and it contains the contraction edge $\{0,3\}$,
the balanced circuit rank of $C_4$ 
with respect to this edge contraction is therefore 0.
\Cref{thm: max corank} states that all cells in $\subdiv_{0,3}(C_4)$ are simplicial.
By \Cref{thm: cell volume}, the volume of each simplicial cell is 2, thus
\[
    \nvol (\adjp_{C_4}) = 2 \cdot 6 = 12,
\]
which agrees with the general volume formula~\cite{ChenDavisMehta2018Counting}
for adjacency polytopes derived from cycles.


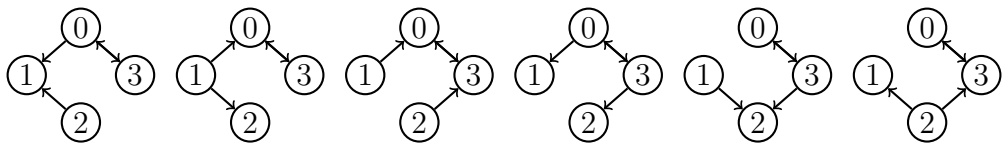
\begin{figure} [h!]
    \centering
    \begin{tikzpicture}[scale=0.9,
        every node/.style={circle,thick,draw,inner sep=1.7},
        every edge/.style={draw,thick}]
        \node (0) at ( 0.5  , .7) {0};
        \node (1) at (-0.3,  0) {1};
        \node (3) at ( 1.3,  0) {3};
        \node (2) at ( 0.5,-.7) {2};
        \path  [->] (3) edge (0);
        \path  [->] (0) edge (3);
        \path  [->] (0) edge (1);
        \path  [->] (2) edge (1);
        
        \begin{scope}[shift={(2.5,0)}]
            \node (0) at ( 0.5  , .7) {0};
            \node (1) at (-0.3,  0) {1};
            \node (3) at ( 1.3,  0) {3};
            \node (2) at ( 0.5,-.7) {2};
            \path  [->] (3) edge (0);
            \path  [->] (0) edge (3);
            \path  [->] (1) edge (0);
            \path  [->] (1) edge (2);
        \end{scope}
        
        \begin{scope}[shift={(5.0,0)}]
            \node (0) at ( 0.5  , .7) {0};
            \node (1) at (-0.3,  0) {1};
            \node (3) at ( 1.3,  0) {3};
            \node (2) at ( 0.5,-.7) {2};
            \path  [->] (3) edge (0);
            \path  [->] (0) edge (3);
            \path  [->] (2) edge (3);
            \path  [->] (1) edge (0);
        \end{scope}
        
        \begin{scope}[shift={(7.5,0)}]
            \node (0) at ( 0.5  , .7) {0};
            \node (1) at (-0.3,  0) {1};
            \node (3) at ( 1.3,  0) {3};
            \node (2) at ( 0.5,-.7) {2};
            \path  [->] (3) edge (0);
            \path  [->] (0) edge (3);
            \path  [->] (0) edge (1);
            \path  [->] (3) edge (2);
        \end{scope}
        
        \begin{scope}[shift={(10.0,0)}]
            \node (0) at ( 0.5  , .7) {0};
            \node (1) at (-0.3,  0) {1};
            \node (3) at ( 1.3,  0) {3};
            \node (2) at ( 0.5,-.7) {2};
            \path  [->] (3) edge (0);
            \path  [->] (0) edge (3);
            \path  [->] (1) edge (2);
            \path  [->] (3) edge (2);
        \end{scope}
        
        \begin{scope}[shift={(12.5,0)}]
            \node (0) at ( 0.5  , .7) {0};
            \node (1) at (-0.3,  0) {1};
            \node (3) at ( 1.3,  0) {3};
            \node (2) at ( 0.5,-.7) {2};
            \path  [->] (3) edge (0);
            \path  [->] (0) edge (3);
            \path  [->] (2) edge (1);
            \path  [->] (2) edge (3);
        \end{scope}
    \end{tikzpicture}
    
    \caption{
        Directed cell subgraphs associated with the 6 cells
        in the subdivision $\subdiv_{0,3}(C_4)$, all of which are simplicial.
    }
    \label{fig:C4 cells}
\end{figure}

\begin{wrapfigure}{r}{0.25\textwidth}
    \centering
        \begin{tikzpicture}[
            every node/.style={circle,thick,draw},
            every edge/.style={draw,thick}]
            \node (0) at ( 0.5  , .6) {0};
            \node (1) at (-0.3,  0) {1};
            \node (4) at ( 1.3  , 0) {4};
            \node (2) at ( 0, -0.9) {2};
            \node (3) at ( 1, -0.9) {3};
            \path (4) edge[color=red,style={double}] (0);
            \path  (1) edge (0);
            \path  (3) edge (4);
            \path  (2) edge (3);
            \path  (2) edge (1);
        \end{tikzpicture}
        \caption{$C_5$is a cycle of 5 nodes.}
        \label{fig: C5}
\end{wrapfigure}
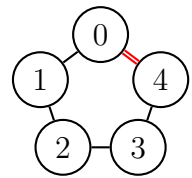
\subsection{A 5-cycle}

In the cycle $C_5$ consisting of 5 nodes shown in \Cref{fig: C5},
we consider the graph contraction with respect to the contraction edge $\{0,4\}$.
Since the only cycle in $C_5$ is an odd cycle, 
and it contains the contraction edge,
it is the unique balanced cycle.
Therefore, the balanced circuit rank of $C_5$ is 1,
and by \Cref{thm: max corank}, we expect the maximum corank of cells to be 1.
Indeed, as shown in \Cref{fig: C5 cells}, all 6 cells are of corank 1
(i.e., each cell subgraph contains a unique undirected cycle).
\Cref{thm: cell volume} states that each of these corank-1 cells
has normalized volume of 5 --- the circumference of the cycle.
Therefore,
\[
    \nvol(\adjp_{C_5}) = 5 \cdot 6 = 30,
\]
which also agrees with the volume formula~\cite{ChenDavisMehta2018Counting}
for adjacency polytopes of cycles.

\begin{figure} [h!]
    \centering
    \begin{tikzpicture}[scale=0.9,
        every node/.style={circle,thick,draw,inner sep=1.7},
        every edge/.style={draw,thick}]
        \node (0) at ( 0.5  , .6) {0};
        \node (1) at (-0.3,  0) {1};
        \node (4) at ( 1.3  , 0) {4};
        \node (2) at ( 0, -0.9) {2};
        \node (3) at ( 1, -0.9) {3};
        \path [->] (4) edge (0);
        \path [->] (0) edge (4);
        \path [->] (0) edge (1);
        \path [->] (1) edge (2);
        \path [->] (3) edge (2);
        \path [->] (4) edge (3);
        
        \begin{scope}[shift={(2.5,0)}]
            \node (0) at ( 0.5  , .6) {0};
            \node (1) at (-0.3,  0) {1};
            \node (4) at ( 1.3  , 0) {4};
            \node (2) at ( 0, -0.9) {2};
            \node (3) at ( 1, -0.9) {3};
            \path [->] (4) edge (0);
            \path [->] (0) edge (4);
            \path [->] (0) edge (1);
            \path [->] (2) edge (1);
            \path [->] (2) edge (3);
            \path [->] (4) edge (3);
        \end{scope}
        
        \begin{scope}[shift={(5.0,0)}]
            \node (0) at ( 0.5  , .6) {0};
            \node (1) at (-0.3,  0) {1};
            \node (4) at ( 1.3  , 0) {4};
            \node (2) at ( 0, -0.9) {2};
            \node (3) at ( 1, -0.9) {3};
            \path [->] (4) edge (0);
            \path [->] (0) edge (4);
            \path [->] (1) edge (0);
            \path [->] (2) edge (1);
            \path [->] (2) edge (3);
            \path [->] (3) edge (4);
        \end{scope}
        
        \begin{scope}[shift={(7.5,0)}]
            \node (0) at ( 0.5  , .6) {0};
            \node (1) at (-0.3,  0) {1};
            \node (4) at ( 1.3  , 0) {4};
            \node (2) at ( 0, -0.9) {2};
            \node (3) at ( 1, -0.9) {3};
            \path [->] (4) edge (0);
            \path [->] (0) edge (4);
            \path [->] (1) edge (0);
            \path [->] (1) edge (2);
            \path [->] (3) edge (2);
            \path [->] (3) edge (4);
        \end{scope}
        
        \begin{scope}[shift={(10.0,0)}]
            \node (0) at ( 0.5  , .6) {0};
            \node (1) at (-0.3,  0) {1};
            \node (4) at ( 1.3  , 0) {4};
            \node (2) at ( 0, -0.9) {2};
            \node (3) at ( 1, -0.9) {3};
            \path [->] (4) edge (0);
            \path [->] (0) edge (4);
            \path [->] (0) edge (1);
            \path [->] (2) edge (1);
            \path [->] (3) edge (2);
            \path [->] (3) edge (4);
        \end{scope}
        
        \begin{scope}[shift={(12.5,0)}]
            \node (0) at ( 0.5  , .6) {0};
            \node (1) at (-0.3,  0) {1};
            \node (4) at ( 1.3  , 0) {4};
            \node (2) at ( 0, -0.9) {2};
            \node (3) at ( 1, -0.9) {3};
            \path [->] (4) edge (0);
            \path [->] (0) edge (4);
            \path [->] (1) edge (0);
            \path [->] (1) edge (2);
            \path [->] (2) edge (3);
            \path [->] (4) edge (3);
        \end{scope}
        
    \end{tikzpicture}
    
    
    \caption{Directed cell subgraphs associated with cells in $\subdiv_{0,1}(C_5)$.}
    \label{fig: C5 cells}  
\end{figure}
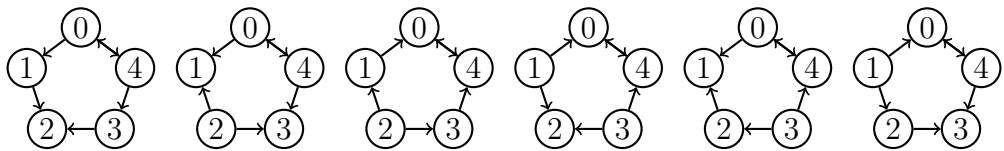

\section{Conclusion}\label{sec:conclusion}

Adjacency polytopes, a.k.a. symmetric edge polytopes,
are convex polytopes associated with connected simple graphs
that have found important applications in several seemingly independent fields.
The set of facets and regular subdivisions of an adjacency polytope
are particularly important in certain applications 
(e.g. the study of algebraic Kuramoto equations).
Recent works established explicit descriptions of the facets and subdivisions
of many families of graphs including trees, cycles, wheels, and bipartite graphs
\cite{ChenDavis2018Toric,ChenDavisMehta2018Counting,dal2019faces}.
The general description for facets and subdivision for arbitrary connected graphs
remains an important open problem.

In this paper, we took one step toward a recursive approach for understanding
the geometric structure of adjacency polytopes associated with large and complex graphs
by considering the effect of an edge-contraction of a graph on the subdivisions
of the corresponding adjacency polytope.
In particular, we showed that an edge-contraction on a graph $G$ naturally induces 
a special regular subdivision of $\adjp_G$ whose cells are in one-to-one correspondence
with facets or product of facets of the adjacency polytope(s) associated with
the smaller resulting graph(s).
On a cell level, we studied 
corank, volume, and affine dependence of individual cells 
as well as the symmetry between combinatorial properties of a cell
and graph-theoretic properties of its corresponding cell subgraph.
This symmetry is captured by the correspondence between two matroids.
We also established the maximum complexity of the cells 
in such edge contraction subdivision,
measured in terms of their coranks, 
to be the balanced circuit rank of the original graph.
Combined with the existing understanding of the facet structures for 
adjacency polytopes associated with trees, cycles, wheels, bipartite graphs, etc.,
these results may shed new light on the regular subdivisions for
more complicated families of graphs.

\section*{Acknoledgements}

This project is an extension of a discussion the authors had with 
Robert Davis,
Alessio D'Alì, Emanuele Delucchi, and Mateusz Michałek.
The authors also thank Robert Davis for reviewing the early draft
and for many suggestions that greatly improved this manuscript.

\bibliographystyle{abbrv}
\bibliography{library}

\begin{thebibliography}{10}

\bibitem{Baillieul1982}
J.~Baillieul and C.~I. Byrnes.
\newblock {Geometric Critical Point Analysis of Lossless Power System Models}.
\newblock {\em IEEE Transactions on Circuits and Systems}, 29(11):724--737, nov
  1982.

\bibitem{Chen2019Directed}
T.~Chen.
\newblock {Directed acyclic decomposition of Kuramoto equations}.
\newblock {\em Chaos: An Interdisciplinary Journal of Nonlinear Science},
  29(9):093101, sep 2019.

\bibitem{Chen2019Unmixing}
T.~Chen.
\newblock {Unmixing the Mixed Volume Computation}.
\newblock {\em Discrete and Computational Geometry}, mar 2019.

\bibitem{ChenDavis2018Toric}
T.~Chen and R.~Davis.
\newblock {A toric deformation method for solving Kuramoto equations}.
\newblock oct 2018.
\newblock \url{http://arxiv.org/abs/1810.05690}.

\bibitem{ChenDavisMehta2018Counting}
T.~Chen, R.~Davis, and D.~Mehta.
\newblock {Counting Equilibria of the Kuramoto Model Using Birationally
  Invariant Intersection Index}.
\newblock {\em SIAM Journal on Applied Algebra and Geometry}, 2(4):489--507,
  jan 2018.

\bibitem{dal2019faces}
A.~D'Al\`{\i}, E.~Delucchi, and M.~Micha\l{}ek.
\newblock Many faces of symmetric edge polytopes, 2019.
\newblock \url{https://arxiv.org/abs/1910.05193}.

\bibitem{DelucchiHoessly2016Fundamental}
E.~Delucchi and L.~Hoessly.
\newblock {Fundamental polytopes of metric trees via parallel connections of
  matroids}.
\newblock dec 2016.
\newblock \url{http://arxiv.org/abs/1612.05534}.

\bibitem{gelfand_discriminants_1994}
I.~M. Gelfand, M.~M. Kapranov, and A.~V. Zelevinsky.
\newblock {\em Discriminants, Resultants, and Multidimensional Determinants}.
\newblock Mathematics: Theory {\&} Applications. Birkh{\"{a}}user Boston, jan
  1994.

\bibitem{Gross2005GraphTheory}
J.~L. Gross and J.~Yellen.
\newblock {\em Graph Theory and Its Applications, Second Edition (Discrete
  Mathematics and Its Applications)}.
\newblock Chapman \& Hall/CRC, 2005.

\bibitem{higashitani2019ARITHMETIC}
A.~Higashitani, K.~Jochemko, and M.~Michałek.
\newblock Arithmetic aspects of symmetric edge polytopes.
\newblock {\em Mathematika}, 65(3):763–784, 2019.

\bibitem{Higashitani2016Interlacing}
A.~Higashitani, M.~Kummer, and M.~Micha{\l}ek.
\newblock {Interlacing Ehrhart Polynomials of Reflexive Polytopes}.
\newblock dec 2016.
\newblock \url{http://arxiv.org/abs/1612.07538}.

\bibitem{Kuramoto1975Self}
Y.~Kuramoto.
\newblock {Self-entrainment of a population of coupled non-linear oscillators}.
\newblock Lecture Notes in Physics, pages 420--422. Springer Berlin Heidelberg,
  1975.

\bibitem{loera_triangulations:_2010}
J.~D. Loera, J.~Rambau, and F.~Santos.
\newblock {\em Triangulations: Structures for Algorithms and Applications}.
\newblock Springer Science {\&} Business Media, aug 2010.

\bibitem{Matsui2011Roots}
T.~Matsui, A.~Higashitani, Y.~Nagazawa, H.~Ohsugi, and T.~Hibi.
\newblock {Roots of Ehrhart polynomials arising from graphs}.
\newblock {\em Journal of Algebraic Combinatorics}, 34(4):721--749, dec 2011.

\bibitem{ohsugi2014}
H.~Ohsugi and T.~Hibi.
\newblock Centrally symmetric configurations of integer matrices.
\newblock {\em Nagoya Math. J.}, 216:153--170, 12 2014.

\bibitem{Ohsugi2012}
H.~Ohsugi and K.~Shibata.
\newblock Smooth fano polytopes whose ehrhart polynomial has a root with large
  real part.
\newblock {\em Discrete and Computational Geometry}, 47(3):624--628, Apr 2012.

\bibitem{Radon1921Mengen}
J.~Radon.
\newblock {Mengen konvexer K{\"o}rper, die einen gemeinsamen Punkt enthalten}.
\newblock {\em Mathematische Annalen}, 83:113--115, 1921.

\bibitem{Rodriguez2002}
F.~Rodriguez-Villegas.
\newblock On the zeros of certain polynomials.
\newblock {\em Proceedings of the American Mathematical Society}, 130, feb
  2002.

\bibitem{Ziegler1995Lectures}
G.~M. Ziegler.
\newblock {\em Lectures on polytopes.}
\newblock Springer-Verlag, New York, 1995.

\end{thebibliography}

\end{document}